\documentclass[12pt]{amsart}
 
\usepackage{amsmath}
\usepackage{amsthm,amssymb,color,comment}
\usepackage{bbm}
\usepackage[colorlinks=true, linkcolor=black, citecolor= black, urlcolor=black, linktocpage=true]{hyperref}
\usepackage[TS1,T1]{fontenc}
\usepackage[utf8]{inputenc}
\usepackage{dsfont}
\usepackage{tikz}
\usepackage{enumitem}

\usepackage{fancyhdr}
\usepackage{todonotes}
\numberwithin{equation}{section}

\newtheorem{theorem}{Theorem}[section]
\newtheorem{proposition}[theorem]{Proposition}
\newtheorem{lemma}[theorem]{Lemma}

\newtheorem{definition}[theorem]{Definition}

\newtheorem{assumption}[theorem]{Assumptions}

\theoremstyle{definition}
\newtheorem{remark}[theorem]{Remark}
\newtheorem{example}[theorem]{Example}

\numberwithin{equation}{section}

\newenvironment{proofof}[1]
{\smallskip\noindent{\textbf{Proof~of~#1.}}
\hspace{1pt}}{\hspace{-5pt}{\nobreak\nobreak\hfill\nobreak
$\square$\vspace{2pt}\par}\smallskip\goodbreak}
\newcommand{\Lip}{\mathbf{Lip}}

\newcommand{\dst}{\displaystyle}

\newcommand{\R}{\mathbb{R}}

\newcommand{\diff}{\mathop{}\!\mathrm{d}}

\makeatletter
\newcommand{\doublewidetilde}[1]{{%
  \mathpalette\double@widetilde{#1}%
}}
\newcommand{\double@widetilde}[2]{%
  \sbox\z@{$\m@th#1\widetilde{#2}$}%
  \ht\z@=.9\ht\z@
  \widetilde{\box\z@}%
}
\makeatother
\author{Christian Düll}
\address{Institute of Mathematics, Heidelberg  University,\\ 69120 Heidelberg, Germany}
\email{duell@math.uni-heidelberg.de}
\thanks{\noindent The research of AMC and CD was supported by the Germany's Excellence Strategy EXC-2181/1 - 390900948 (the Heidelberg STRUCTURES Excellence Cluster) and the European Research Council (ERC) under the European Union's Horizon 2020 research and innovation programme (project PEPS, no. 101071786)}

\author{Piotr Gwiazda}
\address{{\it Piotr Gwiazda:}Institute of Mathematics of Polish Academy of Sciences,\\ Jana i J\k edrzeja \'Sniadeckich 8, 00-656 Warsaw, Poland}
\email{pgwiazda@mimuw.edu.pl}
\thanks{\noindent PG was supported by National Science Center, Poland through project no. 2018/31/B/ST1/02289.}

\author{Anna Marciniak-Czochra}
\address{Institute of Mathematics, Heidelberg  University,\\ 69120 Heidelberg, Germany\\
anna.marciniak@iwr.uni-heidelberg.de}
\email{anna.marciniak@iwr.uni-heidelberg.de}

\author{Jakub Skrzeczkowski}
\address{Faculty of Mathematics, Informatics and Mechanics,\\ University of Warsaw, Stefana Banacha 2, 02-097 Warsaw, Poland}
\email{jakub.skrzeczkowski@student.uw.edu.pl}
\thanks{\noindent JS was supported by the National Science Centre, Poland, grant no. 2017/26/M/ST1/00783}

\begin{document}

\title{\small Structured Population Models on Polish spaces: A unified Approach including Graphs, Riemannian Manifolds and Measure Spaces to describe Dynamics of Heterogeneous Populations}

\begin{abstract}
\tiny This paper presents a mathematical framework for modeling the dynamics of heterogeneous populations.  Models describing local and non-local growth and transport processes, dependent on dynamically changing population structures, appear in a variety of applications such as crowd dynamics, tissue regeneration, cancer development, and coagulation-fragmentation processes. The current body of literature regarding mathematical modeling presents common challenges to mathematicians due to the multiscale nature of the structures that underlie self-organisation and control within complex, heterogeneous systems. In various applications, similar, abstract mathematical concepts arise through problem formulation and the assimilation of mathematical depictions into the language of measure evolution on a multi-faceted state space.  In view of the above observations, we propose an overarching mathematical framework for nonlinear structured population models on abstract metric spaces, which are only assumed to be separable and complete. 
  To achieve this, we exploit the structure of the space of non-negative Radon measures under the dual bounded Lipschitz distance (flat metric), a generalization of the Wasserstein distance that is capable of addressing non-conservative problems. The formulation of models on generic metric spaces facilitates the study on infinite-dimensional state spaces or graphs with a combination of discrete and continuous structures. This opens up exciting possibilities for modeling single cell data, crowd dynamics or coagulation-fragmentation processes.
\end{abstract}

\maketitle
\setcounter{tocdepth}{1}

\noindent
\textbf{Keywords}: structured population model, measure differential equation, Polish spaces, dual bounded Lipschitz distance, flat metric.

\noindent
\textbf{AMS Subject Classification}: 26A16, 28A33, 35B30, 35F10, 35L50, 35R02, 46E27

\section{Introduction}
This paper is devoted to the development of a mathematical framework for structured population models describing the evolution of  systems of individuals that differ with respect to functionally relevant features. The latter may be decoded at different levels of description and defined in a generic metric space. The analysis of the relevant features allows to divide the system into functional subsystems, which may exist at different scales, or to identify continuous transitions in the state space, whose tempo may have significant effects on the emerging dynamics.

The models we study are generic nonlinear structured population models that take into account (see equation \eqref{gen_model_Rd} and formal details given in Section 2):
\begin{itemize}
\item[-] Within-system transitions between different states (directional transitions, diffusion-type processes, etc);
\item[-] Non-linear feedbacks within the system, which can be divided into local and non-local interactions of its individuals;
\item[-] External signalling (environmental, niche signalling, source/inflow terms);
\item[-] Growth/decay processes, which can also be regulated by individual states (local) or by the system (non-local).
\end{itemize}

Motivated by the applications discussed below, we aim to provide a rigorous mathematical framework allowing the formulation of such models in terms of the evolution of Radon measures on a a fairly general metric space. This requires, on the one hand, the generalization of the structured population model \eqref{gen_model_Rd}, which allows the description of the transport of measures on a metric space and, on the other hand, the identification of a suitable structure of the metric space that is sufficient for the well-posedness of the model and, at the same time, covers a broad class of applications. 

\subsection{Biological motivation and the modeling context} 

Our motivation for the development of mathematical approaches to study the dynamics of heterogeneous systems of individuals is primarily related to dynamics and control of cell systems. In addition, we refer to modeling methodologies for crowd dynamics, coagulation-fragmentation processes or other applications concerning emergent dynamics of systems of individuals.

All these applications are linked by the need to develop modeling approaches that account for inter-individual heterogeneity at different, multi-scale levels of description. This challenge for mathematicians has currently been highlighted in the context of several applications, e.g. in recent review articles on perspectives on mathematical modeling of collective dynamics \cite{BellomoReviews} or on mechanistic modeling of stem cell-based systems \cite{DANCIU2023203849}.

{\it Dynamics of heterogeneous cell systems.}
State-of-the-art approaches to study the dynamics of hierarchical cell populations are based on ordinary differential equations  or on stochastic modeling, describing birth-death processes and transitions between different discrete cell states \cite{Spector2017StemCF}. Such compartmental models significantly contributed to the quantification of stem cell traits such as proliferation, self-renewal or quiescence in developmental and regeneration processes and in cancer \cite{Marciniak_Stiehl,Stiehl2019HowTC}. They provided insights into the role of systemic and micro-environmental feedback loops and shed light on heterogeneity of clonal evolution and cell-cell interactions \cite{Busse2016,Stiehl_Biologydirect,Lo2009,Pazdziorek2014,Manesso2013,STUMPF2017268,DANCIU2022103819,Bogdan2014}.

 One of the current challenges for mathematical methods for modeling and model-based data analysis of cell systems concerns data from single cell sequencing (single cell data) that are becoming one of the main tools in the analysis of developmental, regeneration and cancer processes.  In particular, the great advances in the experimental technology for acquiring the cell transcriptomics data allow quantifying the individual cell states in a high-dimensional space of transcriptomes, and thus provide unprecedented insights into the heterogeneity of cell populations
which were previously thought to be homogeneous. Statistical data analysis ordering transcriptomes along a one-dimensional “pseudotime” curve, accompanied by in silico pathway analysis and functional assays, show a transcriptional proximity progression between some functionally different cell states. This, in particular, suggests the existence of some continuous cell state transitions \cite{Theis,Theis16}. 
Such continuous transitions can be modelled using the classical structured population models in the form of transport-type partial differential equations (PDE) or delay differential equations. 
However, the single cell data analysis also points to several aspects that cannot be directly studied in a classical PDE setting, such as co-existence of continuous and discrete (jump) transitions, loops or multiple-structures. Consequently, linking the single cell data to mechanistic mathematical modeling of the underlying processes requires new formulation of the models admitting the envisioned complexity of the cell state space \cite{DANCIU2023203849}.

{\it Crowd dynamics and traffic models.} 
Another topic, which leads to a similar class of models and similar problems in their identification, concerns the dynamics of heterogeneous crowds. In this context, the heterogeneous behaviour arises from psychological or physiological aspects \cite{Bertozzi2015ContagionSI,BellomoReviews,Psych_Crowd2015}, and the functional subsystems may be defined by different walking strategies or walking abilities \cite{Bellomo2019,Bellomo2017,S0218202522500415}.  As discussed in the recent review and perspectives paper Ref. \cite{Bellomo_crowd_review}, the modeling approach needs to take into account features encoded at multiple levels, and the macroscopic dynamics emerge from a collective learning ability that should be taken into account as it progressively modifies the rules of interaction at the individual level.
Depending on the specific scientific questions, identification of the structural variables may lead to hybrid models if both kinetic and macroscopic structures are used \cite{S0218202521400030}. For instance, the velocity direction may be identified using kinetic models while the speed is heuristically modelled as a function of the local population density \cite{Bellomo_Hybrid,Bellomo2016}. Another example is the study of the spread of epidemics in human crowds where the kinetic model for crowd dynamics is coupled with a model of epidemic contagion \cite{S0218202520400126,Kim2021,S0218202521400066}. 

Consequently, the greatest challenge in building a model of structured population dynamics  may be the description of the kinematics of the model, i.e., of the distribution of individual trajectories. Mathematical models of crowd dynamics and the related pedestrian flows \cite{MR2771664,MR3308728} are also formulated in the language of measure evolution \cite{MR3965293,Piccoli_book,100797515}, namely in the form of measure differential equations  \cite{MR4206990,MR3961299,MR4026977}. 
However, so far such models have only been considered in Euclidean spaces which potentially excludes applications based on different geometries or curvatures. In this sense, the framework introduced in Ref. \cite{our_book_ACPJ} and this paper seems to be a promising generalization for both the underlying state space, as well as additional model functions accounting for different effects \cite{MDE_nonlinear_growth}.

{\it Coagulation-fragmentation models}. Another class of structured population models studied in terms of the evolution of measures are coagulation-fragmentation models applied to describe the dynamics of a collection of many particle in physics, chemistry and biology, in particular of oceanic phytoplankton \cite{1029962598,MR4330732,07362990008809704}. Depending on the application, such models account for discrete or continuous structures. A transfer of the model formulation to the space of measures enabled to unify the formerly separated disrete and continuous approaches with the advantageous side effect of allowing for solutions to be mixtures of both \cite{MR4330732}.

\subsection{Choice of the state space motivated by applications}
\label{section:examples}
This section is devoted to discussion of some examples of the problems that go beyond the Euclidean state space. We present two examples that demand a Polish space since they are inherently infinite dimensional and an additional example referring to a proper metric space, which is interesting due to its importance in geometry.
\begin{example}
\label{example:pseudotime distribution}
A natural application of the new setting are structured population models defined on functions or rather trajectories instead of single individuals. In this case one considers a metric space of the form
    \begin{align}
        \label{example trajectory space}
        (S,d)=\left(C^0([0,T];\mathcal{X}), \|\cdot\|_{\infty}\right),
    \end{align}
with some target space $\mathcal{X}$, e.g. $\mathbb{R}^d$ or $\mathbb{N}^d_0$. The space $(S,d)$ is separable and complete but not proper so that the currently available theory cannot be applied. Note that instead of \eqref{example trajectory space}, one can also study the space $S = L^p([0,T];\mathcal{X})$ to account for discontinuous trajectories. 

As an illustration, consider the biological field of inferring the developmental trajectory of cellular differentiation from single-cell RNA-sequencing data (scRNA-seq data). It is assumed that during differentiation a single cell follows a continuous path in some \textbf{gene expression space} $\mathcal{X}$, the so called 
\textbf{individual cell trajectory}. Then all possible trajectories form a metric space of the form \eqref{example trajectory space}. Due to heterogeneity,  cells of a given cell type do not necessarily share a single developmental trajectory, but rather follow a \textbf{cell type specific trajectory distribution} $\mu\in \mathcal{M}^+(C^0([0,T];\mathcal{X}))$, or $\mu\in \mathcal{P}(C^0([0,T];\mathcal{X}))$ if we are only interested in relative counts. Due to biological processes,  external influences or fast epigenetic mutations \cite{Clairambault2015} it can be assumed that the trajectory distribution is eligible to change. 
\end{example}

\begin{example}
For the next application assume that we do not consider a cell type specific trajectory distribution (see Example \ref{example:pseudotime distribution}) but rather take its average/ median to get the so called \textbf{pseudotime trajectory} $\gamma:[0,T]\to\mathcal{X}$. By taking the average, many cells roughly follow the differentiation process prescribed by $\gamma$. In this example we are interested in the pseudotemporal development of a specific gene. In other words, for each pseudotime point $t$ the gene is expressed with some intensity, e.g. absolute gene counts, in all cells following $\gamma$. Cell heterogeneity rather implies a gene expression distribution among all cells, $\mu_t\in \mathcal{M}^+(S)$ with $S=\mathbb{N}_0$ or $S=\mathbb{R}^+$. 

In the case of several genes $i=1,...,N$ along the  trajectory $\gamma$, there is a distribution $\mu_t^i\in \mathcal{M}^+(S)$ for each pseudo\-time point $t$ and each gene $i$. Then the evolution of all gene distributions along the trajectory $\gamma$ can be described by a measure on $\mathcal{M}^+(S)$, e.g. in the discrete case here by a measure of the form
    \begin{align*}
        M_t^N=\sum_{i=1}^N\alpha_iD_{\mu_t^i}\in \mathcal{M}^+(\mathcal{M}^+(S)),
    \end{align*}
where we use $D_{\mu}$ to denote the Dirac measure on $\mathcal{M}^+(S)$ located in $\mu$. In the case of infinitely many genes more elaborate measures on $\mathcal{M}^+(S)$ could occur, e.g. measures without atoms. Although it is possible to model such objects by products of measure spaces, the number of genes $N$ may be very large and thus it might be convenient to consider the limit of $M_t^N$ in $\mathcal{M}^+(\mathcal{M}^+(S))$ as $N\to \infty$. The latter can be an arbitrary element of $\mathcal{M}^+(\mathcal{M}^+(S))$. 
\end{example}

\begin{example}
\label{example:riemannian geometry}
The setting of this paper can also be used to consider the evolution of measures on specific Riemannian manifolds. Let $S$ be a finite dimensional smooth manifold and let $g$ be a Riemannian metric so that the pair $(S,g)$  is a connected and geodesically complete Riemannian manifold. We refer to Chapter 6 in Ref. \cite{Lee} for precise definitions and more information on Riemannian geometry. Well-known examples of connected and geodesically complete Riemannian manifolds are the torus, the sphere or the hyperbolic space. If we equip $S$ with the corresponding Riemannian distance function $d_R$, which measures the length of the shortest path between two points, $(S,d_R)$ becomes a metric space. According to the theorem of Hopf-Rinow \cite{jost}, the space $(S,d_R)$ is even proper (see Remark \ref{rem:proper}) so that we can conclude separability and completeness by Proposition A.15 in Ref. \cite{our_book_ACPJ}.
In particular, the setting of Section~\ref{section:general case} can be used to formulate structured population models on $(S,d_R)$. Such models appear for example in Ref. \cite{MR4206990,MR3965293,rossi2016control}.
\end{example}

 \subsection{Analytical setting and current challenges}

 A promising framework for analysis of structured population systems is offered by the recently developed mathematical theory of PDE models defined in the space of Radon measures \cite{ackleh2020mathematical,MR4330732,MR4126768,Carrillo_Colombo_Gw_U,MR3150768,evers2015mild,MR3507552,farkas2019asymptotic,GLMC,GMC,GMT_Measures_under_Flat_Norm,MR4066016,MR4060809,MR4219146,MR4292766} and a related approach of measure differential equations \cite{MR4206990,MDE_nonlinear_growth,MR3961299,MR4026977}. To account for both discrete and continuous cell distributions as well as a possible non-Euclidean structure of the state space such as graphs, one needs to go beyond $\R^d$ and define models on a reasonably generic metric space. One choice of the latter would be proper metric spaces (see Remark \ref{rem:proper}) which include networks, graphs or complete Riemannian manifolds (see Example \ref{example:riemannian geometry}).  However, if the underlying metric space $S$ is a normed vector space, then basic functional analysis implies that a proper space $S$ is finite dimensional. In order to avoid this limitation, we claim that separability and completeness are already sufficient to define a well-posed model on $S$, which allows considering infinite-dimensional state spaces. 

An additional advantage of formulating the measure on separable and complete metric spaces, so-called Polish metric spaces, is that it provides a link to concepts and methods developed in stochastic modeling. In particular, the theory of concentrating Feller operators offers promising results on asymptotic analysis of measure solutions \cite{MR2287895,MR2271485,MR2995659,Thieme_2022}. For example, the asymptotics of transport-type equations can be linked to ergodic properties of Markov processes \cite{MR2857021}. A connection between the two fields has been recently used to show  nonexpansiveness of a semigroup of solutions in a transport distance, corresponding to the flat metric \cite{MR4342006,MR4347325}.

From a modeling perspective, a unified theory of well-posedness for models defined on a broad class of state spaces enables the analysis of state space structure through inverse problems using mechanistic models of biological processes and associated data. Understanding the structure of the state space is currently one of the main challenges in single cell data analysis. To this end, model-based approaches can be a helpful tool, complementing the topology-based \cite{TDA_breast_cancer,TDA_cancer,ScData_TDA}, geometric \cite{PoincareMap} or statistical approaches currently used for data analysis.  A recent example of model-based analysis of cell state space allowed a comparison of a continuous and a graph-based state space structure \cite{ChoRockne2022}. However, the latter example lacks a uniform setting so that the two alternative models have to be studied separately. We show how to circumvent such a limitation by using models formulated in Radon measures on Polish metric spaces.  The proposed setting is sufficiently general so that it can simultaneously account for discrete and continuous distributions. The latter is desired not only in models of cell differentiation but also, for example, in modeling coagulation-fragmentation processes \cite{MR4330732,MR4126768}. 

\subsection{Structure of the paper}
In Section \ref{section:case Rd} we start with a short summary of the measure formulation on $\R^d$ to motivate the generalized notion of solution on Polish  metric spaces.  Section \ref{section:measure basics} introduces the basic measure theoretic concepts which are necessary for understanding the paper. The main results are presented in Section \ref{section:general case}. In particular, we first generalize the flow of vector fields to Polish metric spaces and then, formulate and analyze the models in the linear and nonlinear case. The model analysis is concluded by showing its equivalence to the  PDE models in the $\R^d$ case. To streamline the presentation, the proofs are shifted to Sections \ref{section:proof_linear} (linear model), \ref{section:proof_nonlinear} (nonlinear model) and \ref{section:proof_of_consistency} (consistency with the PDE model). The article is complemented by an overview on measure theory in Section \ref{appendix: measure theory}, which, although it does not contain any new results, enhances the presentation of the main article and ensures its self-containment.

\section{From a PDE formulation to a generalized model}
\label{section:case Rd}
We start with a PDE formulation of the structured population model, since it allows linking process-based model ingredients with an implicit formula for the model solution which, in turn, can be defined as the generalized model whose formulation does not use derivatives. Moreover, we use the $\R^d$ case to sketch the concept of analysis of structured population models in terms of Radon measures. 

We begin with clarifying the notation. By $\mu_{\bullet}$ we denote a family of measures $\{\mu_t \}_{t\in[0,T]}$ while $\mu_t$ represents the evaluation of $\mu$ at a given time point $t$. We consider the following nonlinear model defined on $[0,T]\times\mathbb{R}^d $
\begin{align}\label{gen_model_Rd}
\left\{\begin{array}{ll}
\partial_t \mu_t +\nabla_x\cdot (b(t,x,\mu_t) \mu_t)
 &=  c(t,x,\mu_t)\mu_t + \int_{\mathbb{R}^d} \eta(t,x,\mu_t)(y)  \diff\mu_t(y) + N(t,\mu_t)   ,
\\
\mu_0 &= \nu, \end{array}\right.
\end{align}
where $\nu\in \mathcal{M}^+(\mathbb{R}^d)$ and 
 \begin{align*}
 \begin{array}{rr}
     b: [0,T]\times\mathbb{R}^d \times \mathcal{M}^+(\mathbb{R}^d)  \to \mathbb{R}^d,\hfill & c:[0,T]\times\mathbb{R}^d \times \mathcal{M}^+(\mathbb{R}^d)  \to \mathbb{R},\hfill\\
     \eta:[0,T]\times \mathbb{R}^d \times \mathcal{M}^+(\mathbb{R}^d) \to \mathcal{M}^+(\mathbb{R}^d),\hspace{0.3cm}&N:[0,T]\times \mathcal{M}^+(\mathbb{R}^d) \to \mathcal{M}^+(\mathbb{R}^d).
 \end{array}
	\end{align*}
Function $c$ represents a \textbf{growth term}, $\eta$ can be interpreted as \textbf{spread of heterogeneity} (such as a \textbf{mutation kernel}), $N$ is a \textbf{state-independent influx} and $b$ denotes the \textbf{flow of the vector field} which is responsible for the transformation dynamic of the individual states. In this section, we skip the conditions on the model functions and instead refer to Assumption \ref{ass_special_caseRd}. Based on the weak formulation of the problem, we introduce measure solutions to model \eqref{gen_model_Rd}.
\begin{definition}\label{gen_model_Rd_def}
A family of measures $\mu_{\bullet}:=\{\mu_t\}_{t\in[0,T]}\subset \mathcal{M}^+(\mathbb{R}^d)$ is a \textbf{measure solution} to \eqref{gen_model_Rd} provided $t \mapsto \mu_t$ is narrowly continuous and for any test function $\varphi \in C^1([0,T]\times \mathbb{R}^d)\cap W^{1,\infty}([0,T]\times\mathbb{R}^d)$,
\begin{align}\label{Nonlin_weak_gen_ver_chap3_Rd}
&  \int_{\mathbb{R}^d} \varphi(T,x)  \diff \mu_T(x)
  -\int_{\mathbb{R}^d}  \varphi(0,x)  \diff \mu_0(x) =
  \int_0^T \int_{\mathbb{R}^d} \partial_t \varphi(t,x)  \diff \mu_t(x)  \diff t \nonumber \\
 & \phantom{ = }+  \dst \int_0^T \int_{\mathbb{R}^d} \big( \nabla_x  \varphi(t,x) \cdot b(t,x,\mu_t)  + \varphi(t,x) \, c(t,x,\mu_t) \big)   \diff\mu_t(x)  \diff t  \nonumber \\
 &\phantom{ = }+ \int_0^T \int_{\mathbb{R}^d} \left(\int_{\mathbb{R}^d} \varphi(t,y)  \diff\!\left[\eta(t,x,\mu_t) \right](y) \right)  \diff\mu_t(x)  \diff t \nonumber \\
 & \phantom{ = }+
 \int_0^T \int_{\mathbb{R}^d} \varphi(t,x)  \diff N(t, \mu_t)(x)  \diff t .
\end{align}
  \end{definition}
  
 \begin{remark}
 In view of Theorem \ref{thm: flat norm equivalent to narrow convergence}, narrow continuity of a map $\mu_{\bullet}:[0,T]\to \mathcal{M}^+(S)$ is equivalent to continuity with respect to the flat metric. Thus, the flat metric provides a suitable framework to investigate measure solutions.
 \end{remark}
 
\begin{theorem}[Well-posedness of the model on $\R^d$]
\label{thm:Rd_Main} 
Under Assumption \ref{ass_special_caseRd}, there exists a unique Lipschitz continuous solution $ \mu_{\bullet}:[0,T] \to ({\mathcal M^+}(\mathbb{R}^d), \rho_F)$ to model \eqref{gen_model_Rd} with initial measure $\mu_0 \in {\mathcal M^+}(\mathbb{R}^d)$. Moreover, the solution is continuous with respect to time, initial measure as well as model functions.
\end{theorem}

A complete proof of Theorem \ref{thm:Rd_Main} is so far only available in case $S=\R^+$, see Section 2 in Ref. \cite{our_book_ACPJ}. 
Here, we sketch briefly the main ideas of the proof for $\R^d$ as they will serve as a road map for the analysis of the generalized model later on.

First, we set $N=0$ and consider a linear version of \eqref{gen_model_Rd}. 
In this case the solution can be constructed by duality  theory and the method of characteristics \cite{DiPernaLions,PerthameBook}. In particular, for $t\in [0,T]$ and $\psi\in C^1(\mathbb{R}^d)\cap W^{1,\infty}(\mathbb{R}^d)$ the corresponding dual problem is of the form
 \begin{align}
\label{lin_dual_without_N}
    \left\{\begin{array}{lll}
    \partial_{\tau}   \varphi_{\psi, t}   +  b \cdot \nabla_x \,\varphi_{\psi, t} +  c \, \varphi_{\psi, t}
   +  \int_{\mathbb{R}^d} \varphi_{\psi, t}(\tau,y)  \,\mathrm{d}\!\left[\eta(\tau,x) \right](y)
& =  0 &\mbox{in} \;\;\;  [0,t] \times\mathbb{R}^d,\\
  \varphi_{\psi, t}(t,\cdot)
& =  \psi & \mbox{in} \;\;\; \mathbb{R}^d.
\end{array}\right.
    \end{align}
The method of characteristics together with Banach Fixed Point Theorem yields a unique solution $\varphi_{\psi,t}\in C^1([0,t]\times \mathbb{R}^d)\cap W^{1,\infty}([0,t]\times \mathbb{R}^d)$ with representation formula
\begin{align*}
    \begin{split}
    &\varphi_{\psi,t}(\tau,x)=\,\psi(X_b(t-\tau,x))e^{\int_{\tau}^tc(r,X_b(r-\tau,x))\diff r}\\
   &+\int_{\tau}^t\int_{\mathbb{R}^d}\varphi_{\psi,t}(s,y)\diff[\eta(s,X_b(s-\tau,x))](y)e^{\int_{\tau}^sc(r,X_b(r-\tau,x))\diff r}\diff s.
   \end{split}
    \end{align*}

\noindent Here $X_b(t,\tau,x)$ denotes the flow generated by the vector field $b$ and which is defined as the unique solution of the ODE
\begin{align*}
\partial_t X_b(t, \tau, x) = b(t,X_b(t, \tau, x)) \quad \quad X_b(\tau, \tau, x) = x. 
\end{align*}
Defining $\mu_t$ via
\begin{align}
    \label{lin_model_Rd_def_mu_t}
    \int_{\mathbb{R}^d}\psi(x)\diff \mu_t(x)=\int_{\mathbb{R}^d}\varphi_{\psi,t}(0,x)\diff \mu_0(x)
    \end{align}
it can be shown that $\mu_t$ satisfies a semigroup property and is actually the unique solution to the primal problem with $N=0$. In the case with $N$ it is not possible to go to the dual (i.e. adjoint) problem directly as the PDE is not linear in $\mu_t$ anymore. So instead of \eqref{lin_model_Rd_def_mu_t}, we define
    \begin{align}    \label{lin_model_Rd_def_mu_t_with_N}
    \int_{\mathbb{R}^d}\psi(x)\diff \mu_t(x)=\int_{\mathbb{R}^d}\varphi_{\psi,t}(0,x)\diff \mu_0(x)+\int_0^t \int_{\mathbb{R}^d}\varphi_{\psi,t}(\tau,x)\diff [N(\tau)](x)\diff \tau,
    \end{align}
where $\varphi_{\psi,t}$ is still the solution to the dual problem \eqref{lin_dual_without_N} with $N=0$. The implicit representation formula for $\varphi_{\psi,t}$ translates to an integral representation of $\mu_t$ 
\begin{align}
\label{implicit representation Rd}
	&\int_{\R^d} \psi(x)  \diff\mu_t(x) = \int_{\R^d} \psi(X_b(t,0,x)) e^{\int_0^t c(s,\, X_b(s,\,0,\,x))  \diff  s}  \diff\mu_0(x) \nonumber \\
	&+\int_0^t \int_{\R^d} \int_{\R^d} \psi(X_b(t,\tau,y)) e^{\int_{\tau}^t c(s,\, X_b(s,\,\tau,\,y)) \diff  s}  \diff[\eta(\tau,x)](y)  \diff\mu_{\tau}(x)  \diff\tau \nonumber \\
	&+\int_0^t  \int_{\R^d} \psi(X_b(t,\tau,x)) e^{\int_{\tau}^t c(s,\, X_b(s,\,\tau,\,x)) \diff  s}  \diff[N(\tau)](x)   \diff\tau.
\end{align}
 
\noindent With representation \eqref{implicit representation Rd}, the existence as well as the continuity result in the linear case follow. In order to construct solutions to the nonlinear problem where the model functions depend on the measure argument, we apply Banach Fixed Point Theorem to a suitable operator.

\begin{remark}
Observe that the implicit representation \eqref{implicit representation Rd} does not require the concept of derivatives and can thus be defined on any metric space, giving \eqref{implicit representation Rd} a major advantage over the direct transport equation \eqref{gen_model_Rd_def}.
\end{remark}

\section{Measure Theory Tools}
\label{section:measure basics}
In this section we shortly introduce the basic measure theoretic concepts which are necessary for this paper. As usual, continuous functions from a space $X$ to a space $Y$ are denoted by $C^0(X,Y)$. If $Y=\mathbb{R}$, we omit the second argument and simply write $C^0(X)$. We will consider a metric space $(S,d)$ which is assumed to be \textbf{Polish}, i.e. separable and complete. Throughout this paper, we  work with signed Borel measures on $(S,d)$ and we refer to Chapter 3.1 in Ref. \cite{Folland.1984} for basic definitions. In view of Hahn-Jordan Decomposition Theorem, any signed measure $\mu$ has a unique representation $\mu=\mu^+-\mu^-$, where $\mu^+, \mu^-$ are nonnegative measures. If both $\mu^+(S), \mu^-(S) < \infty$, $\mu$ is said to be \textbf{finite} and the \textbf{space of all finite Borel measures} on $(S,d)$ is denoted by $\mathcal{M}(S)$. The corresponding \textbf{cone of nonnegative measures} is defined as $\mathcal{M}^+(S)=\left\{\mu \in \mathcal{M}(S)\mid \mu \geq 0\right\}$, where the partial ordering "$\leq$" on $\mathcal{M}(S)$ is given setwise, i.e. $\mu\leq \nu$ iff $\mu(A)\leq \nu(A)$ for all $A\in \mathcal{B}(S)$.

There are several norms on $\mathcal{M}(S)$ such as the well-known total variation norm $\|\cdot\|_{TV}$
$$
\| \mu \|_{TV} = \mu^+(S) + \mu^-(S).
$$
Even though $(\mathcal{M}(S),\|\cdot\|_{TV})$ is a Banach space, the total variation norm is not suitable for the applications in this paper as the corresponding topology is too strong. In particular, measure solutions will not be continuous with respect to initial conditions (see Examples 2.1 \& 2.2 in Ref. \cite{our_book_ACPJ}). Instead, we use the weaker \textbf{flat norm} (or \textbf{dual bounded Lipschitz norm}, \textbf{Fortet-Mourier norm}) 
	\begin{align}
		\label{defeq:flatnorm}
		\|\mu\|_{BL^*}:= \sup\left\{\int_S  \psi \diff\mu \mid \psi \in BL(S), \|\psi\|_{BL} \leq 1\right\}
	\end{align} 
and the corresponding \textbf{(flat) metric} $\rho_F$. The space of test functions is given by the bounded Lipschitz functions
	\begin{align*}
		BL(S)=\{f\in C^{0}(S)\mid \|f\|_{BL}<\infty\},
	\end{align*}
with the  (semi-)norms 
\begin{align*}
	\|f\|_{BL}:=\max \left \{ \|f\|_{\infty},|f|_{\Lip} \right\}, \hspace{0.4cm}\|f\|_{\infty}=\underset{x\in S}{\sup}\,|f(x)|, 
	\hspace{0.4cm}|f|_{\Lip}=\underset { x\neq y}{\sup } \,\frac {|f(x)-f(y)|}{d(x,y)}.
	\end{align*}
If we consider a sequence of measures $(\mu_n)_{n\in \mathbb{N}}\subset \mathcal{M}^+(S)$ converging to a measure $\mu$ with respect to the flat metric, then a compatible notion of convergence is the so-called \textbf{narrow convergence} (see Theorem \ref{thm: flat norm equivalent to narrow convergence}) which denotes convergence in duality with all continuous and bounded functions, i.e.
		\begin{align*}
			\lim_{n\to \infty} \int_{S}  \psi(x) \diff\mu_n(x)= \int_{S}  \psi(x) \diff\mu(x) \qquad\forall \psi \in C_b(S).
		\end{align*}
Accordingly, a map $\mu_{\bullet}:[0,T]\to \mathcal{M}(S)$, $t\mapsto \mu_t$ is called \textbf{narrowly continuous}  if $\mu_{\bullet}$ is continuous with respect to narrow convergence, i.e.  
			\begin{align*}
				\lim_{s \to t} \int_{S} \! \psi(x) \diff\mu_s(x)= \int_{S} \! \psi(x) \diff\mu_t(x) \qquad \forall \psi \in C_b(S) \text{ and } \forall t\in [0,T].
			\end{align*}
			
\noindent Theorem \ref{thm:Rd_Main} implies that solutions to model \eqref{gen_model_Rd} form a Lipschitz semigroup. The generalization of this property is based on the \textbf{push-forward of a measure} $\mu\in \mathcal{M}^+(S)$ under a measurable map $T:(X,\mathcal{A})\to (Y, \mathcal{B})$ given by $T_{\#}\mu(B) := \mu(T^{-1}(B))$
and the \textbf{change of variable formula}
\begin{align}\label{change_form_pf_chap3}
\int_Y f  \diff(T_{\#}\mu) = \int_X f \circ T  \diff\mu
\end{align}
which holds if either integral is well-defined, see Section 3.6 in Ref. \cite{Bogachev1.2007}. By choosing $f$ in \eqref{change_form_pf_chap3} to be a characteristic function we see that the push-forward measure $T_{\#}\mu$ is uniquely defined on $(Y,\mathcal{B})$.

 \section{Model formulation on a Polish metric space}
 \label{section:general case}
Equipped with the motivation of Section \ref{section:case Rd} and the measure theory basics of Section~\ref{section:measure basics}, we turn to structured population models which are formulated on a Polish metric space $(S,d)$. As abstract metric spaces do not possess a linear structure, the concept of derivatives is unavailable and thus also the method of characteristics - used to construct the solution of the  transport process in \eqref{gen_model_Rd}. Inspired by the results of Section \ref{section:case Rd}, we choose an alternative approach and start with generalizing the flow of the vector field $b$.
 
\subsection{Two-parameter families of homeomorphisms}\label{2par_pf}
\begin{definition}\label{Flow} 
Let $(S,d)$ be a separable metric space. For every pair $(t,\tau)\in \mathbb{R}^2$ we define a \textbf{two-parameter family of bi-Lipschitz homeomorphisms} $X(t,\tau,\cdot): S \rightarrow S$ such that
\begin{itemize}
\item[(i)] the map $t \rightarrow X(t,\tau,\cdot)$ is uniformly continuous, i.e. 
$$
\sup_{\tau \in \mathbb{R}}\sup_{x \in S} d(X(t_1,\tau,x), X(t_2,\tau,x))\leq \omega_X(|t_1-t_2|)
$$ 
with a modulus of continuity $\omega_X$ ($\omega_X: [0,\infty] \to [0,\infty]$ is continuous in $0$ and satisfies $\omega_X(0) = 0$).
\item[(ii)] $\lim_{t\rightarrow \tau}X(t,\tau,\cdot)=\mathrm{Id}$ in $C^0(S)$, i.e.
$\lim_{t\rightarrow \tau} \sup_{x \in S} d(X(t,\tau,x), x) = 0,$
\item[(iii)] $X(t_2,\tau,\cdot)=X(t_2,t_1,\cdot) \circ X(t_1,\tau,\cdot)$ for every $t_1,t_2 \in \mathbb R$,
\item[(iv)] there exists a locally bounded function $L_X(\cdot)$ with
	\begin{align*}
		&d(X(t,\tau,x_1),X(t,\tau,x_2))\leq L_X(t-\tau) \, d(x_1,x_2)\qquad \text{ and }\\
		&d(X^{-1}(t,\tau,x_1), X^{-1}(t,\tau,x_2))\leq L_X(t-\tau) \, d(x_1,x_2).
	\end{align*}
\end{itemize}
We may also restrict possible values of time arguments $\tau$ and $t$ to some given interval $[0,T]$ so that $t, \tau \in [0,T]$.
\end{definition}

\begin{remark} 
Bi-Lipschitz homeomorphisms may be seen as generalized flows of vector fields on $(S,d)$. In the notation $X(t,\tau,x)$, $t$  and $\tau$ can be interpreted as a current time and as an initial time point, respectively, while $x$ represents initial position. This includes backward flows when $t < \tau$.  
\end{remark}

\subsection{Linear structured population model on a metric space}\label{linear_general_metric_section}
\label{subsection:linear_model}
With the methods of Section \ref{2par_pf}, we aim at a generalization of model \eqref{gen_model_Rd} for the state space $(S,d)$ being separable and complete. As a metric space the state space  
possibly lacks linear structure, so we can not formulate an explicit PDE model on $(S,d)$. Nevertheless, we can proceed analogously to Section \ref{section:case Rd} by introducing model functions in the linear case
\begin{align*}
 \begin{array}{rr}
     c:[0,T]\times\mathbb{R}^d \times \mathcal{M}^+(\mathbb{R}^d)  \to \mathbb{R},\hspace{0.3cm}& \eta: [0,T] \times S \to \mathcal M^+(S),\hfill\\
     N: [0,T] \to \mathcal M^+(S),\hfill &X: [0,T] \times [0,T] \times S \to S.
 \end{array}
	\end{align*}
The interpretation of $c,\eta$ and $N$ is similar to their counterparts in Section \ref{section:case Rd}, namely $c$ represents a \textbf{growth term}, $\eta$ is a \textbf{spread of heterogeneity} and $N$ is a \textbf{state-independent influx}. Function $X$ is the \textbf{generalization of the flow of the vector field $b$} in the sense of Definition \ref{Flow}. Furthermore, we assume:
\begin{assumption}[Model functions]\label{assumptions_general_model}
The model functions $c$, $\eta$, $N$ and $X$ satisfy
\begin{itemize}
\item[(i)] $c \in L^1((0,T); BL(S))$,
\item[(ii)]  $\eta \in  L^1\left((0,T); BL(S;\mathcal M^+(S))\right)$,
\item[(iii)]  $N \in  L^1\left((0,T); \mathcal M^+(S)\right)$,
\item[(iv)]  $X$ is a bi-Lipschitz homeomorphism on $(S,d)$.
\end{itemize}
\end{assumption}
To explain the technical assumptions on the model functions, we elaborate  assumption (ii) from a computational point of view. For fixed $t \in (0,T)$ and $x \in S$, $\eta(t,x)$ is a measure in $\mathcal{M}^+(S)$, whereas for fixed $t$ the map $x \mapsto \eta(t,x) \in \mathcal{M}^+(S)$ is bounded and Lipschitz continuous, i.e.  
$$
\| \eta(t,\cdot) \|_{BL(S;\mathcal{M}^+)} := \max\left\{\sup_{x \in S} \| \eta(t,x) \|_{BL^*}, |\eta(t,\cdot)|_{\Lip}\right\} < \infty, 
$$
where the Lipschitz constant is given by
$$
 |\eta(t,\cdot)|_{\Lip} = \sup_{ x_1 \neq x_2} \frac{\rho_F(\eta(t,x_1),\eta(t,x_2))}{d(x_1,x_2)}.
$$
Finally, $\| \eta(t,\cdot) \|_{BL(S;\mathcal{M}^+)}$ is required to be integrable in time, i.e.
$$
\| \eta \|_{L^1_T(BL(S;\mathcal{M}^+))} = 
\int_0^T \| \eta(t,\cdot) \|_{BL(S;\mathcal{M}^+)}  \diff t
$$
is finite and thus for a.e. $t \in (0,T)$ the expression $\| \eta(t,\cdot) \|_{BL(S;\mathcal{M}^+)}$ is also finite. 

If $S= \R^d$ we can directly use the vector field $b$ instead of generalizing it to $X$. In Assumption \ref{ass_special_caseRd} we  provide assumptions on $b$ in the nonlinear case.

Now we can introduce the linear structured population model on $(S,d)$. The notion of solution is motivated by the implicit integral representation \eqref{implicit representation Rd}.

\begin{definition} [Generalized Model]\label{solution_linear_general_defn}
We say that a family of measures $\mu_{\bullet}:=\{\mu_t\}_{t\in[0,T]} \subset \mathcal{M}^+(S)$ is a \textbf{(generalized) solution to the linear structured population model} on $(S,d)$ with initial measure $\mu_0 \in \mathcal{M}^+(S)$ and model functions $(c,\eta, N,X)$ satisfying Assumptions \ref{assumptions_general_model}, if $ t \mapsto \mu_t$ is narrowly continuous and $\mu_t$ satisfies
	\begin{align}\label{LinearModelMetric}
		&\mu_t=X(t,0,\cdot)_{\#}\left(\mu_0(\cdot) e^{\int_0^t c(s,\,X(s,\,0,\cdot)) \diff s}\right)  \nonumber \\
	 	&+\int_0^tX(t,\tau,\cdot)_{\#}\left(\int_S \left[\eta (\tau, y) (\cdot)\right] \diff \mu_{\tau} (y) e^{\int_{\tau}^t c(s,\,X(s,\,\tau,\cdot)) \diff s}\right) \diff \tau \nonumber\\
	 	&+ \int_0^tX(t,\tau,\cdot)_{\#}\left(N(\tau)(\cdot) e^{\int_{\tau}^t c(s,\,X(s,\,\tau,\cdot)) \diff s}\right) \diff \tau.
	 \end{align}
 \end{definition}

\noindent Both integrals in the second part of formula \eqref{LinearModelMetric} should be understood in the Bochner sense. 

\begin{proposition}[Rigorous definition of integrals]
\label{proposition:rigorous definition of integrals}
The inner and outer Bochner integrals appearing in  \eqref{LinearModelMetric} are well-defined measures in $\mathcal{M}^+(S)$. 
\end{proposition}

\begin{remark}
Consider the conservative case when $c,\eta,N\equiv 0$, i.e. the initial measure $\mu_0$ is just moved according to the map $X(t,\cdot,0)$. In this case the solution $\mu_t$ satisfies $\mu_t=X(t,\cdot,0)_{\#}\mu_0$ which is exactly the solution to the optimal transport problem with transport map $X(t,\cdot,0)$.
So from an optimal transport perspective formula \eqref{LinearModelMetric} provides a substantial generalization to the simple optimal transport case.
\end{remark}

With a fixed point argument applied to a suitable chosen operator, we are then able to prove the following
\begin{theorem}[Well-posedness of the linear model]\label{existence_gener_lin_thm}
Under Assumption \ref{assumptions_general_model}, there exists a unique solution to the linear problem \eqref{LinearModelMetric} with initial measure $\mu_0\in \mathcal{M}^+(S)$.  Furthermore, if $\mu$, $\nu$ are solutions with initial measures $\mu_0$ and $\nu_0$ respectively, then there exists a constant $C$ depending on $X$ and $c$ such that
    \begin{align*}
    \rho_F(\mu_t, \nu_t) \leq C \rho_F(\mu_0, \nu_0) \, e^{C \int_0^t \| \eta(\tau,\cdot) \|_{BL(S;\mathcal{M}^+)}  \diff\tau},
    \end{align*}
i.e. the solutions are continuous with respect to initial conditions.\\
Similarly, if  $\mu^{\eta, N, X, c}$ and $\mu^{\tilde \eta, \tilde N, \tilde X, \tilde c}$ are the two solutions of \eqref{LinearModelMetric} with the same initial condition $ \mu_0 \in \mathcal{M}^+(S)$ but different sets of model functions $(\eta, N, X, c)$ and $(\tilde \eta, \tilde N, \tilde X, \tilde c)$, then
	\begin{align*}
	&\rho_F\left(\mu^{\eta, N, X, c}_t, \mu^{\tilde \eta, \tilde N, \tilde X, \tilde c}_t\right) \leq C_M \int_0^t \left[\sup_{y \in S} \rho_F\left(\eta(\tau,y),\tilde \eta(\tau, y)\right) + \rho_F\left(N(\tau),\tilde N (\tau) \right) \right] \diff\tau  \\ 
	&   +C_M\int_0^t\|c(\tau, \cdot)- 	\tilde c(\tau, \cdot) \|_{\infty}  \diff\tau +C_M \sup_{0 \leq \tau_1 \leq \tau_2 \leq t} \|X(\tau_2, \tau_1, \cdot) - \tilde X(\tau_2, \tau_1, \cdot)  \|_{\infty},
	\end{align*}
i.e. the solutions are continuous with respect to model functions.
\end{theorem}

\subsection{The nonlinear model}\label{Chap33_nonlinprob}
\label{subsection:nonlinear_model}
Next, we turn to the nonlinear version of \eqref{LinearModelMetric} which involves the following model functions
\begin{align*}
    \begin{array}{rr}
    c: [0,T] \times S \times \mathcal M^+(S) \to \mathbb{R},\hfill&
    \eta: [0,T] \times S  \times \mathcal M^+(S) \to \mathcal M^+(S),\hfill\\
    N: [0,T] \times \mathcal M^+(S) \to \mathcal M^+(S),\hspace{0.2cm} &
    X: [0,T] \times [0,T] \times S \times \left([0,T] \to \mathcal M^+(S) \right) \to S.
    \end{array}
\end{align*}
\noindent Compared to the setting of the linear model, there is an additional measure-valued argument which allows the model functions to depend on the solution itself. 

At first glance it might not be clear why the function $X$ takes as an argument the whole map from $[0,T] \to \mathcal M^+(S)$ rather than an individual measure from $\mathcal M^+(S)$. Recall that the map $X$ is the generalization of the flow of a vector field on $\mathbb{R}^d$ given by nonlinear ODEs of the form
\begin{equation}\label{eq_nonlinear_ODE_example_motiv_X_arg}
\partial_t X_b(t,\tau,x) = b(t,X_b(t,\tau,x), \mu_t), \quad \quad \quad X_b(\tau,\tau,x) = x,
\end{equation}
where $b: [0,T] \times \mathbb{R}^d \times \mathcal M^+(\R^d) \to \mathbb{R}^d$ is a vector field and $\mu_{\bullet}:[0,T]\to \mathcal M^+(\R^d)$ is a measure-valued map. Integrating \eqref{eq_nonlinear_ODE_example_motiv_X_arg} from $\tau$ to $t$ yields
    \begin{align*}
    X_b(t,\tau,x) = x + \int_{\tau}^t b(s,X_b(s,\tau,x), \mu_s) \diff s,
    \end{align*}
so that $X_b(t,\tau,x)$ may depend on values of $\mu_s$ for time arguments $\tau \leq s \leq t$. 

\begin{assumption}[Nonlinear model]\label{assumptions_general_model_non}
The model functions $\eta, N$ and $c$ should satisfy the following conditions.
\begin{itemize}
\item[(i)] For any $\mu \in \mathcal{M}^+(S)$, the model functions $\eta(\cdot,\cdot,\mu), N(\cdot, \mu)$ and $c(\cdot,\cdot,\mu)$ satisfy Assumption \ref{assumptions_general_model}. Moreover, we have uniform bounds
\begin{equation}\label{cond_suff_to_est_unif_bounds}
\int_0^T \sup_{\mu \in \mathcal{M}^+(S)} \Big[\|c(\tau,\cdot,\mu) \|_{\infty} +  \sup_{y \in S} \|\eta(\tau,y,\mu) \|_{BL^*} +   \|N(\tau,\mu) \|_{BL^*}\Big]   \diff\tau < \infty,
\end{equation}
and for any $R > 0$ 
$$
\int_0^T \sup_{\| \mu \|_{BL^*} \leq R} \Big[|c(\tau,\cdot,\mu) |_{\Lip} +  \|\eta(\tau,\cdot,\mu) \|_{BL(S;\mathcal{M}^+)}\Big]   \diff\tau < \infty.
$$
\item[(ii)] For any narrowly continuous map $\mu_{\bullet}: [0,T] \to \mathcal{M}^+(S)$, $X(t,\tau,x,\mu_{\bullet})$ is a two-parameter family of  bi-Lipschitz homeomorphisms in $t$ and $\tau$. Moreover, for any $R$ and any $\mu_{\bullet}$ with $\sup_{t \in [0,T]} \| \mu_{t} \|_{BL^*} \leq R$, there exists a modulus of continuity $\omega_{X,R}:[0,\infty] \to [0,\infty]$ and a locally bounded function $L_{X,R}:\mathbb{R} \to \mathbb{R}$ satisfying
\begin{align}\label{conditions_for_nonlinear_flow}
&\sup_{\tau \in \mathbb{R}}\sup_{x \in S} d(X(t_1,\tau,x,\mu_{\bullet}), X(t_2,\tau,x,\mu_{\bullet}))\leq \omega_{X,R}(|t_1-t_2|), \nonumber\\
		&d(X(t,\tau,x_1,\mu_{\bullet}),X(t,\tau,x_2,\mu_{\bullet}))\leq L_{X,R}(t-\tau) \,d(x_1,x_2), \phantom{\sup_{\tau \in \mathbb{R}}} \nonumber \\
		&d(X^{-1}(t,\tau,x_1,\mu_{\bullet}), X^{-1}(t,\tau,x_2,\mu_{\bullet}))\leq L_{X,R}(t-\tau)\, d(x_1,x_2). \phantom{\sup_{\tau \in \mathbb{R}}} 
\end{align}
This means that the properties of a bi-Lipschitz homeomorphism are satisfied uniformly for each ball in the space of measures.
\item[(iii)] For any $R > 0$, there exists $L_{R,c} \in L^1(0,T)$ so that for $\|\mu\|_{BL^*}, \|\nu\|_{BL^*} \leq R$ 
$$
\|c(t,\cdot,\mu) - c(t,\cdot,\nu)\|_{\infty} \leq L_{R,c}(t) \rho_F(\mu, \nu).
$$ 
\item[(iv)] For any $R > 0$, there exists $L_{R,\eta}\in L^1(0,T)$ so that for $\|\mu\|_{BL^*}, \|\nu\|_{BL^*} \leq R$
$$
\sup_{y\in S} \rho_F\left(\eta(t,y,\mu), \eta(t,y,\nu) \right) \leq L_{R,\eta}(t) \, \rho_F(\mu, \nu).
$$  
\item[(v)] For any $R > 0$, there exists $L_{R,N}\in L^1(0,T)$ so that if $\|\mu\|_{BL^*}, \|\nu\|_{BL^*} \leq R$
$$
\sup_{y\in S} \rho_F\left(N(t,\mu), N(t,\nu) \right) \leq L_{R,N}(t)\, \rho_F(\mu, \nu).
$$  
\item[(vi)] For any $R > 0$, there exists $L_{R,X} \in L^1(0,T)$ so that for $\sup_{t\in [0,T]} \|\mu_t\|_{BL^*}$, $\sup_{t\in [0,T]} \|\nu_t\|_{BL^*} \leq R$ 
$$
\|X(t_2,t_1, \cdot,\mu_{\bullet}) -  X(t_2,t_1, \cdot,\nu_{\bullet})\|_{\infty} \leq \int_{t_1}^{t_2} L_{R,X}(\tau) \, \rho_F(\mu_{\tau}, \nu_{\tau})  \diff\tau.
$$  
\end{itemize}
\end{assumption}

Condition \eqref{cond_suff_to_est_unif_bounds} allows to bound solutions in the total variation norm so that all other bounds can be assumed to hold only locally. The following lemma shows that (ii) and (vi) in Assumption \ref{assumptions_general_model_non} indeed generalize the situation known for classical flows of vector fields.

\begin{definition}\label{def_soltononlinearmodel_gen} We say that $\mu_{\bullet}: [0,T] \to \mathcal{M}^+(S)$ is a \textbf{generalized solution to the  nonlinear structured population model on $(S,d)$} with initial measure $\mu_0 \in\mathcal{M}^+(S)$ and model functions $\eta, X, c, N$  satisfying Assumption \ref{assumptions_general_model_non}, if $\mu_{\bullet}$ is narrowly continuous and satisfies
	\begin{align}\label{NonlinearModelMetric}
	&\mu_t=X(t,0,\cdot,\mu_{\bullet})_{\#}\left(\mu_0(\cdot) e^{\int_0^t c(s,\, X(s,\,0,\cdot,\,\mu_{\bullet}),\, \mu_{s}) \diff s}\right)  \phantom{\int}\nonumber \\
	 &+\int_0^tX(t,\tau,\cdot, \mu_{\bullet})_{\#}\left(\int_S \left[\eta (\tau, y, \mu_{\tau}) (\cdot)\right]  \diff \mu_{\tau} (y) e^{\int_{\tau}^t c(s,\,X(s,\,\tau,\cdot,\, \mu_{\bullet}), \, \mu_{s}) \diff s}\right) \diff \tau \nonumber\\
	 	&+ \int_0^tX(t,\tau,\cdot, \mu_{\bullet})_{\#}\left(N(\tau, \mu_{\tau})(\cdot) e^{\int_{\tau}^t c(s,\,X(s,\,\tau,\cdot,\, \mu_{\bullet}),\, \mu_{s}) \diff  s}\right) \diff\tau.
	 \end{align}
 \end{definition}
To prove existence and uniqueness of solutions the nonlinear model is reduced to a linear version where the measure argument is fixed. We then conclude by a fixed point argument.
\begin{theorem}[Well-posedness of the nonlinear model]\label{uniq_theorem_nonlin_general}
Let $\mu_0 \in \mathcal{M}^+(S)$ be an initial measure and suppose  the model functions $(\eta,N, X,c)$ satisfy Assumption \ref{assumptions_general_model_non}. Then there exists a unique solution of the nonlinear structured population model on $(S,d)$ in the sense of Definition \ref{def_soltononlinearmodel_gen}. Moreover, the solution is continuous with respect to initial conditions and model functions.  More precisely, let $(\eta, N, X, c)$ and $(\tilde \eta, \tilde N, \tilde X, \tilde c)$ be two sets of model functions which both satisfy Assumption \ref{assumptions_general_model_non} and let $\mu^{\eta, N, X, c}_t$ and $\mu^{\tilde \eta, \tilde N, \tilde X, \tilde c}_t$ be the  solutions to the corresponding nonlinear structured population models. Then the following estimate holds
	\begin{align}\label{cont_mf_gen_nonlinear_chap3}
	&\rho_F\big(\mu^{\eta, N, X, c}_t,\, \mu^{\tilde \eta, \tilde N, \tilde X, \tilde c}_t\big)  \leq   C_M  \int_0^t \sup_{y \in S} \sup_{\nu \in \mathcal{M}^+(S)} \rho_F\left(\eta(\tau,y,\nu),\tilde 
	\eta(\tau, y,\nu)\right)  \diff\tau  \nonumber \\ 
	 &+\, C_M  \int_0^t \sup_{\nu \in \mathcal{M}^+(S)} \rho_F\left(N(\tau,\nu), 
	\tilde N(\tau, \nu)\right)  \diff\tau \nonumber \\
	&+\,C_M  \int_0^t \sup_{\nu \in \mathcal{M}^+(S)} \|c(\tau, \cdot, \nu)- \tilde c(\tau, \cdot,\nu) \|_{\infty} \diff\tau \nonumber \\
	&+\,C_M \sup_{\nu_{\bullet} \in C^0([0,t];\, \mathcal{M}^+(S))} \sup_{0 \leq \tau_1 \leq \tau_2 \leq t} \|X(\tau_2, \tau_1, \cdot, \nu_{\bullet}) - \tilde X(\tau_2, \tau_1, \cdot, 
	\nu_{\bullet})  \|_{\infty}.
	\end{align}
Similarly, if $\mu_t^{(1)}$ and $\mu_t^{(2)}$ solve the  nonlinear structured population model with identical model functions $(\eta, N, X, c)$  but different initial conditions $\mu_0^{(1)}$ and $\mu_0^{(2)}$, then there is a constant $C_M$ such that
\begin{equation}\label{cont_ic_gen_nonlinear_chap3}
\rho_F(\mu_t^{(1)}, \mu_t^{(2)}) \leq C_M \, \rho_F(\mu_0^{(1)},\mu_0^{(2)}) .
\end{equation}
\end{theorem}

\subsection{Equivalence of the generalized model with the PDE Model in $\R^d$ case}\label{sect:gen_setting_in_Rd}
\label{subsection:consistency}
\label{subsection:generalized model is pde model}
We verify that in the Euclidean domain $S=\R^d$ the generalized solution and the measure solution to \eqref{gen_model_Rd} are equivalent. In this case, instead of bi-Lipschitz homeomorphisms, it is sufficient to consider the flow of the vector field $b$ defined by the unique solution of the ODE
\begin{align}\label{general_ODE_chap3}
\partial_t X_b(t, \tau, x, \mu_{\bullet}) = b(t,X_b(t, \tau, x, \mu_{\bullet}), \mu_t) \quad \quad X_b(\tau, \tau, x, \mu_{\bullet}) = x. 
\end{align}
In fact, it can be shown, that under the following assumptions, $X_b$ satisfies (ii) and (vi) in Assumption \ref{assumptions_general_model_non} and is thus an explicit example for a bi-Lipschitz homeomorphism.

\begin{assumption}[Special case $\R^d$]
\label{ass_special_caseRd}
We replace hypotheses (ii) and (vi) on $X$ in Assumption \ref{assumptions_general_model_non} by the following assumptions on $b: [0,T] \times \mathbb{R}^d \times \mathcal{M}^+(\mathbb{R}^d) \to \mathbb{R}^d$:
\begin{itemize}
\item[(i)] For any $\mu \in \mathcal{M}^+(\mathbb{R}^d)$, the function $b(\cdot,\cdot, \mu)$ is in $L^1((0,T); BL(\mathbb{R}^d;\mathbb{R}^d))$. Moreover, for any $R>0$
$$
\int_0^T \sup_{\| \mu \|_{BL^*} \leq R} \|b(\tau,\cdot,\mu) \|_{BL}  \diff\tau < \infty.
$$ 
\item[(ii)] For any $R > 0$, there exists $L_{R,b}$ in $L^1(0,T)$ so that if $\mu_{\bullet}$, $\nu_{\bullet}$ are narrowly continuous and $\|\mu\|_{BL^*}, \|\nu\|_{BL^*} \leq R$, then 
$$
\|b(t,\cdot,\mu) - b(t,\cdot,\nu)\|_{\infty} \leq L_{R,b}(t) \rho_F(\mu, \nu).
$$  
\end{itemize}
The model functions $c$, $\eta$ and $N$ are still assumed to satisfy Assumption \ref{assumptions_general_model_non}.
\end{assumption}
\begin{theorem}\label{reprimpliesmeasuresol}
Suppose Assumptions \ref{ass_special_caseRd} hold true. 
Let $\mu_0 \in \mathcal{M}^+(\mathbb{R}^d)$ and let $X_b(t, \tau, x, \mu_{\bullet})$ be defined as the unique solution of the ODE
\begin{equation*}\label{general_ODE_chap34_copied}
\partial_t X_b(t, \tau, x, \mu_{\bullet}) = b(t,X_b(t, \tau, x, \mu_{\bullet}), \mu_t) \quad \quad X_b(\tau, \tau, x, \mu_{\bullet}) = x. \end{equation*} 
Then the generalized solution given by representation formula \eqref{NonlinearModelMetric} provides a measure solution to \eqref{gen_model_Rd} in the sense of Definition \ref{gen_model_Rd_def} and is thus the unique measure solution to the problem.
\end{theorem}

\section{Proof of the linear model}
\label{section:proof_linear}
We begin with a proof of Proposition \ref{proposition:rigorous definition of integrals} which guarantees the well-posedness of the Bochner integrals involved in the solution formula \eqref{LinearModelMetric}.

\begin{proofof}{Proposition \ref{proposition:rigorous definition of integrals}}
For the integrals appearing in \eqref{LinearModelMetric} to be well-posed, we require the target space to be a separable and complete vector space. Therefore, we will work in the Banach space
    \begin{align*}
        E := \overline{\mathcal{M}(S)}^{\|\cdot\|_{BL(S)^*}},
    \end{align*}
i.e. the closure of $\mathcal{M}(S)$ as a subspace of $BL(S)^*$. As closure of a separable space (cf. Theorem \ref{thm:M(S)_separablel}) $E$ is separable and complete as a closed subset of $BL(S)^*$. Thus, according to Section V  \S 4 in Ref. \cite{Yosida}, the concepts of weak and strong measurability coincide on $E$. Its dual space has been characterized in Theorem \ref{thm:Predual_bounded_Lipschitz} to be $E^*=BL(S)$.  Hence, to check strong measurability of the map 
$S \ni x \mapsto \eta(t,x) \in E$,
 we shall verify that for any $f \in E^*$ and a.e. $t \in (0,T)$, the map  
 $$ S \ni x \mapsto \langle f,\eta(t,x)\rangle_{E^*,E}
 $$
   is measurable where $\langle\cdot, \cdot\rangle_{E^*, E}$ denotes the usual dual pairing. However, when $\eta$ is assumed to be in $BL(S; \mathcal M^+(S))$, one can even prove its Lipschitz continuity as for any $x_1, x_2 \in S$
	\begin{align*}
		\big|\big\langle f,\eta(t,x_1) - \eta(t,x_2)\big\rangle_{E^*,E}\big| \leq& \,\|f\|_{E^*} \, \rho_F\left(\eta(t,x_1), \eta(t,x_2)\right)   
		 \\ \leq \, \|f\|_{E^*} \, |\eta(t,\cdot)|_{\Lip}\, d(x_1,x_2). 
	\end{align*}
We conclude that $S \ni x \mapsto \eta(t,x)$ is strongly measurable. Furthermore, Theorem~1 in Section V \S 5 in Ref. \cite{Yosida} implies that for a.e. $t \in (0,T)$ the map $x \mapsto \eta(t,x)$ is $\mu_{t}$-integrable since
\begin{equation}\label{norm_bound_intB}
\int_S \|  \eta(t,x) \|_{BL^*}  \diff\mu_{t}(x)  \leq \| \eta(t,\cdot) \|_{BL(S;\mathcal{M}^+)} \|\mu_{t}\|_{BL^*} < \infty
\end{equation}
and narrow continuity of $[0,T]\ni \tau \mapsto \mu_{\tau}$ yields that $\sup_{\tau \in [0,T]}\|\mu_{\tau}\|_{BL^*} < \infty$. Thus for a.e. $\tau \in (0,T)$, $I(\tau) = \int_S \left[\eta (\tau, y) (\cdot)\right]  \diff \mu_{\tau} (y)$ is a well defined Bochner integral on $E$. The computation in \eqref{norm_bound_intB} also implies that
$$
\left\|I(t) \right\|_{BL^*} = \left\|\int_S \eta(t,x)  \diff\mu_{t}(x) \right\|_{BL^*} \leq \| \eta(t,\cdot) \|_{BL(S;\mathcal{M}^+)} \|\mu_{t}\|_{BL^*}.
$$
Up to now we have only proven that $I(\tau) \in E$. Using Theorem \ref{satz:respresentation of positive functionals}, we show that it is actually a nonnegative measure. First, we claim that $I(\tau)$ is an element of $BL(S)^*_+$ (see Definition \ref{def:BL^*_+}). Indeed, let $f \in BL(S)$ with $f \geq 0$. As Bochner integrals commute with bounded, linear operators (see \cite[V \S 5, Corollary 2]{Yosida}), we compute for any $\nu \in \mathcal{M}^+(S)$
    \begin{align*}
    \left\langle \int_{S} \eta(t,x)(\cdot)  \diff\nu(x), f \right\rangle_{BL^*,\, BL} = 
    \int_{S} \left\langle \eta(t,x)(\cdot),f \right\rangle_{BL^*,\, BL}  \diff\nu(y) 
    \geq 0,
    \end{align*}
since $\eta(t,x) \in \mathcal{M}^+(S)$. With Theorem \ref{satz:respresentation of positive functionals} we obtain  $I(\tau) \in \mathcal{M}^+(S)$.  Next we focus on the outer integral
$$
J = \int_0^t X(t,\tau,\cdot)_{\#}\left[I(\tau) (\cdot)  e^{\int_{\tau}^t c(s,\,X(s,\tau,\cdot))  \diff s}\right]  \diff \tau.
$$
One can check that $\tau \mapsto X(t,\tau,\cdot)_{\#}\left[I(\tau) (\cdot)  e^{\int_{\tau}^t c(s,\,X(s,\,\tau,\cdot))  \diff s}\right]$ is weakly measurable. Furthermore, the map
$$
\tau \mapsto \left\langle f, X(t,\tau,\cdot)_{\#}\left[I(\tau) (\cdot)  e^{\int_{\tau}^t c(s,\,X(s,\,\tau,\cdot))  \diff s}\right] \right \rangle_{E^*, E}
$$ is continuous and a similar argument as for $I(\tau)$ shows that $J$ is also a nonnegative Radon measure on $S$. The second integral in \eqref{LinearModelMetric} can be analyzed analogously.
\end{proofof}

Before we look at well-posedness theory for solutions in Definition \ref{solution_linear_general_defn}, we want to understand how we can compute integrals of the form $\int_S \psi(x)  \diff\mu_t(x)$ for measures $\mu_t$ given by the implicit equation \eqref{LinearModelMetric}. 
 To this end, we decouple the left and the right hand side of  \eqref{LinearModelMetric}.

\begin{lemma}
Consider a narrowly continuous map $\nu_{\bullet}:[0,T] \to \mathcal{M}^+(S)$ and let $\mu_0 \in \mathcal{M}^+(S)$. For $t\in[0,T]$ we define a family of measures $\mu_t$ by
	\begin{align*}
		&\mu_t:=X(t,0,\cdot)_{\#}\left(\mu_0(\cdot) e^{\int_0^t c(s,\,X(s,\,0,\cdot)) \diff s}\right)  \nonumber \\
		 &+\int_0^tX(t,\tau,\cdot)_{\#}\left(\int_S \left[\eta (\tau, y) (\cdot)\right]  \diff \nu_{\tau} (y) e^{\int_{\tau}^t c(s,\,X(s,\,\tau,\cdot)) \diff s}\right) \diff \tau \\
		 &+ \int_0^tX(t,\tau,\cdot)_{\#}\left(N(\tau)(\cdot) e^{\int_{\tau}^t c(s,\,X(s,\,\tau,\cdot)) \diff s}\right) \diff \tau \nonumber.
	\end{align*}
Then for any $\psi \in BL(S)$ the integral with respect to $\mu_t$ can be computed as follows
	\begin{align}\label{how_to_integrate}
	&\int_S \psi(x)  \diff\mu_t(x) = \int_S \psi(X(t,0,x)) e^{\int_0^t c(s,\, X(s,\,0,\,x))  \diff  s}  \diff\mu_0(x) \nonumber  \\
	&+\int_0^t \int_S \int_S \psi(X(t,\tau,y)) e^{\int_{\tau}^t c(s,\, X(s,\,\tau,\,y)) \diff  s}  \diff[\eta(\tau,x)](y)  \diff\nu_{\tau}(x)  \diff\tau \nonumber \\
	&+\int_0^t  \int_S \psi(X(t,\tau,x)) e^{\int_{\tau}^t c(s,\, X(s,\,\tau,\,x)) \diff  s}  \diff[N(\tau)](x)   \diff\tau.
	\end{align}
\end{lemma}

\begin{proof}
We look at each term of $\mu_t$ separately to compute $\int_S \psi(x)  \diff\mu_t(x)$. For the first term we apply the change of variable formula \eqref{change_form_pf_chap3} for push-forward measures which leads to the first term on the right-hand side of \eqref{how_to_integrate}. Before we consider the second term we note that $\psi \in BL(S)$ implies $\psi \in E^*$ by Theorem \ref{thm:Predual_bounded_Lipschitz}. Thus, we can use that linear operations commute with the Bochner integral \cite{Yosida} which leads to
	\begin{align*}
		&\int_S \psi(y)  \diff\!\left[\int_0^tX(t,\tau,\cdot)_{\#}\left(\int_S \left[\eta (\tau, x) (\cdot)\right] \diff \nu_{\tau} (x) e^{\int_{\tau}^t c(s,\,X(s,\,\tau,\cdot)) \diff s}\right) \diff \tau 
		\right](y)  \\ 
	&	=\int_0^t \int_S \psi(y)  \diff\!\left[X(t,\tau,\cdot)_{\#}\left(\int_S \left[\eta (\tau, x) (\cdot)\right]  \diff \nu_{\tau} (x) e^{\int_{\tau}^t c(s,\,X(s,\,\tau,\cdot)) \diff s}\right) \right]
		(y)  \diff	\tau  \\ 
		&  =\int_0^t \int_S \psi(X(t,\tau,y))  \diff\!\left[\int_S \left[\eta (\tau, x) (\cdot)\right]  \diff \nu_{\tau} (x) e^{\int_{\tau}^t c(s,\,X(s,\,\tau,\cdot)) \diff  s}\right](y)  \diff \tau  \\ 
		& =\int_0^t \int_S \int_S \psi(X(t,\tau,y)) e^{\int_{\tau}^t c(s,\,X(s,\,\tau,\, y)) \diff s}  \diff\!\left[\eta (\tau, x)\right] (y)  \diff \nu_{\tau} (x)  \diff\tau. 
\end{align*} 
The third term follows similarly. 
\end{proof}

To streamline the following computations, we collect some regularity and boundedness results from Lemma 3.17 in Ref. \cite{our_book_ACPJ}.
\begin{lemma}\label{easier_notation}
Suppose functions $X$ and $c$ satisfy Assumption \ref{assumptions_general_model}. For all test functions $\psi$ with $\|\psi\|_{BL}\leq 1$.
    \begin{enumerate}
        \item [(i)] The map $[\tau,T] \ni t\mapsto G^{\psi, X,c}_{\tau,t}(x):=\psi(X(t,\tau,x))e^{\int_{\tau}^tc(s,X(s,\tau,x))\diff s}$ is uniformly continuous, i.e. there exists a modulus of continuity $\omega_{X,c}$ such that
        	\begin{align*}
	    	\left| G^{\psi, X,c}_{\tau,t_2}(x)-G^{\psi, X,c}_{\tau,t_1}(x)\right| \leq \omega_{X,c} \left(|t_2 - t_1| \right).
	    \end{align*}
	    \item [(ii)] For any $0\leq \tau \leq t \leq T$, the map
        $S \ni x\mapsto G^{\psi, X,c}_{\tau,t}(x)$ is in BL(S) and $\|G^{\psi, X,c}_{\tau,t}\|_{BL}$  can be bounded independently of $\psi$  by 
        	\begin{align}\label{constant_C_X,c}
        	\begin{split}
        	 C^{X,c}_{\tau,t} := \, e^{\int_{0}^{T} \|c(s, \cdot)\|_{\infty} \diff  s} \left[L_X(t-\tau) + \int_{\tau}^{t} |c(s, \cdot)|_{\Lip}\, L_X(s - \tau)  \diff  s  \right]
	        \end{split}, 
    	\end{align}
    	where he function $L_X$ has been introduced in Definition \ref{Flow}.
    \end{enumerate}
\end{lemma}

The proof of existence and uniqueness of solutions in the sense of Definition \ref{solution_linear_general_defn} is based on Banach Fixed Point Theorem applied to an operator built on \eqref{LinearModelMetric}. In order to find a suitable complete metric which allows our operator to be a contraction, we generalize the so-called {\it Bielecki norm}. Motivated by Ref. \cite{kwapisz}, for nonnegative $f \in L^1(0,T)$ and $\lambda > 0$, a new metric on $C^0([0,T]; \mathcal{M}^+(S))$ is given by
\begin{equation}\label{L1Bielecki}
\rho_{\lambda, f}(\mu_{\bullet}, \nu_{\bullet}) = \sup_{t \in [0,T]}\left[ e^{-\lambda \int_0^t f(u)  \diff u } \rho_F(\mu_t, \nu_t) \right]\!.
\end{equation}
We remark that for $f = 1$, \eqref{L1Bielecki} reduces to a metric related to the standard Bielecki norm.

According to Theorem \ref{thm:M+_complete_wrt_flat}, the space  $(\mathcal{M}^+(S), \rho_F)$ is complete and consequently $C^0([0,T]; \mathcal{M}^+(S))$ equipped with the natural metric
\begin{equation}\label{eq:natural_metric_cmeas}
\rho_{\infty}(\mu_{\bullet}, \nu_{\bullet}) := \sup_{t \in [0,T]} \, \rho_F(\mu_t, \nu_t).
\end{equation}
is complete as well. As the metric $\rho_{\lambda,f}$ is equivalent to $\rho_{\infty}$, the space $(C^0([0,T]; \mathcal{M}^+(S)),\rho_{\lambda, f})$ is also  complete. 
Now that there is a suitable complete metric space, we construct our solution as a fixed point of the map
    \begin{align*}
    T: C^0([0,T]; \mathcal{M}^+(S)) \to C^0([0,T]; \mathcal{M}^+(S))
    \end{align*}
defined as follows. Given a measure-valued map $\mu_{\bullet} \in C^0([0,T]; \mathcal{M}^+(S))$, we define $(T\mu_{\bullet})_{\bullet}$ by the right-hand side of \eqref{LinearModelMetric}, i.e.
	\begin{align}\label{eq:def_of_map_fixed_point_32}
	&\left(T\mu_{\bullet}\right)_t=X(t,0,\cdot)_{\#}(\mu_0(\cdot) e^{\int_0^t c(s,\,X(s,\,0,\cdot)) \diff s})  \nonumber \\
	 	&+\int_0^tX(t,\tau,\cdot)_{\#}\left(\int_S \left[\eta (\tau, y) (\cdot)\right] \diff \mu_{\tau} (y) e^{\int_{\tau}^t c(s,\,X(s,\,\tau,\cdot)) \diff s}\right) \diff \tau \nonumber\\
	 	&+ \int_0^tX(t,\tau,\cdot)_{\#}\left(N(\tau)(\cdot) e^{\int_{\tau}^t c(s,\,X(s,\,\tau,\cdot)) \diff s}\right) \diff \tau.
	 \end{align}
It can be shown that $(T\mu_{\bullet})_{\bullet}:[0,T]\to \mathcal{M}^+(S)$ is continuous. The proof is mainly based on the uniform continuity of the map $t\mapsto \psi(X(t,\tau,x))e^{\int_{\tau}^tc(s,X(s,\tau,x))\diff s}$ (see Lemma \ref{easier_notation} (i)). The details are not hard but rather technical, so we skip the proof. 

We are now able to prove the first part of Theorem \ref{existence_gener_lin_thm} concerning existence and uniqueness for the linear model. The claimed continuous dependencies are considered separately.

\begin{proofof}{Existence and Uniqueness in Theorem \ref{existence_gener_lin_thm}}
Consider two measure-valued maps $\mu^{(1)}_{\bullet}, \mu^{(2)}_{\bullet} \in C^0([0,T]; \mathcal{M}^+(S))$ with an identical initial measure $ \mu_0~\in~\mathcal{M}^+(S)$. According to Assumption \ref{assumptions_general_model}  for all $\psi \in BL(S)$ and for all $t \in [0,T]$ the map
    \begin{align}
    \label{boundedness_of_integral_over_eta}
        S \ni x \mapsto \int_S \psi(y)  \diff[\eta(t,x)](y)  
    \end{align}
is in $BL(S)$ with norm bounded by $\|\psi\|_{BL}\, \| \eta(t,\cdot)\|_{BL(S;\mathcal{M}^+)}$. Thus, we estimate with \eqref{how_to_integrate} and \eqref{constant_C_X,c} 
	\begin{align*}
		&\rho_F\left((T\mu_{\bullet}^{(1)})_t, (T\mu_{\bullet}^{(2)})_t\right)  \\
		& \leq\sup_{\|\psi\|_{BL} \leq 1} \int_0^t \int_S \int_S \psi(X(t,\tau,y))\, e^{\int_{\tau}^{t} c(s, X(s,\,\tau,\,y)) \diff  s} \diff[\eta(\tau,x)](y)  \diff\!\left[\mu^{(1)}_{\tau} - \mu^{(2)}
		 _{\tau} \right](x)  \diff\tau  \\ 
		& \leq \, \int_0^t  2\, C^{X,c}_{\tau,t}\, \| \eta(\tau,\cdot) \|_{BL(S;\mathcal{M}^+)}\, \rho_F\left(\mu^{(1)}_{\tau},\mu^{(2)}_{\tau} \right)  \diff\tau.
	\end{align*}
Let $C = \sup_{0 \leq t \leq T} \sup_{0 \leq \tau \leq t} C^{X,c}_{\tau,t}$. Then
$$
\rho_F((T\mu^{(1)}_{\bullet})_t, (T\mu^{(2)}_{\bullet})_t) \leq
C\, \int_0^t  \| \eta(\tau,\cdot) \|_{BL(S;\mathcal{M}^+)} ~ \rho_F\left(\mu^{(1)}_{\tau},\mu^{(2)}_{\tau} \right)  \diff\tau.
$$
For some $\lambda >0 $ and $f \in L^1(0,T)$ to be chosen later, we apply the generalized Bielecki norm $\rho_{\lambda, f}$ defined in \eqref{L1Bielecki}
	\begin{align*}
	&\rho_{\lambda, f}(T\mu^{(1)}_{\bullet}, T\mu^{(2)}_{\bullet}) = \sup_{t \in [0,T]}\left[ e^{-\lambda \int_0^t f(u)  \diff u } \rho_F((T\mu^{(1)})_t, (T\mu^{(2)})_t) \right]   \\
	 & \leq C \sup_{t \in [0,T]} e^{-\lambda \int_0^t f(u)  \diff u } \int_0^t  \| \eta(\tau,\cdot) \|_{BL(S;\mathcal{M}^+)} ~	
	 \rho_F\left(\mu^{(1)}_{\tau},\mu^{(2)}_{\tau} \right)  \diff\tau  \\
	&  \leq C \rho_{\lambda, f}(\mu^{(1)}_{\bullet}, \mu^{(2)}_{\bullet}) \sup_{t \in [0,T]} e^{-\lambda \int_0^t f(u)  \diff u } \int_0^T  \| \eta(\tau,\cdot) \|_{BL(S;\mathcal{M}^+)} ~e^{\lambda \int_0^{\tau} f(u) 
	   \diff u }  \diff\tau  \\
	    & = \frac{C}{\lambda} \rho_{\lambda, f}(\mu^{(1)}_{\bullet}, \mu^{(2)}_{\bullet}) \left(1 - e^{-\lambda \int_0^T f(u)  \diff u } \right) \leq \frac{C}{\lambda}\rho_{\lambda, f}(\mu^{(1)}_{\bullet}, \mu^{(2)}_{\bullet}).
	\end{align*}
In the last line we chose $f(u) = \| \eta(u,\cdot) \|_{BL(S;\mathcal{M}^+)}$ and observed that the map $t \mapsto \lambda f(t)e^{\lambda \int_0^t f(u)  \diff u }$ is the derivative of $t \mapsto e^{\lambda \int_0^t f(u)  \diff u }$ due to Lebesgue Differentiation Theorem (see Theorem 6 in Appendix E in Ref. \cite{EvansBook}). Setting $\lambda = 2\,C$ finishes the proof.
\end{proofof}

Now we prove the continuous dependency of solutions on initial conditions. The proof for continuous dependency on model functions is based on similar computations combined with triangle inequalities.
Before we start with the initial conditions, we prove an {\it \`{a} priori} bound on the measure solution.

\begin{lemma}\label{bound_on_TV_norm_gen_model}
Suppose $[0,T] \ni t \mapsto \mu_t \in \mathcal{M}(S)$ solves \eqref{LinearModelMetric} in the sense of Definition \ref{solution_linear_general_defn}. Then $\|\mu_t \|_{BL^*}$ is uniformly bounded for all $t\in [0,T]$.
\end{lemma}

\begin{proof}
Formula \eqref{how_to_integrate} with $\psi = 1$ yields 	\begin{align*}
	&\| \mu_t \|_{BL^*} \leq\, \| \mu_0 \|_{BL^*} \, e^{\int_{0}^t \| c(s,\cdot) \|_{\infty} \diff s}  \\
	&\,+ \int_0^t  \left(\sup_{x\in S} \|\eta (\tau, x)\|_{BL^*} \|\mu_\tau \|_{BL^*} +\|N (\tau)\|_{BL^*} \right) \emph{}	
	e^{\int_{\tau}^t \| c(s,\cdot) \|_{\infty} \diff s} \diff \tau 
	\end{align*}
so that the claim follows by Gr\"{o}nwall's inequality.
\end{proof}

\begin{proofof}{Continuous dependence on Initial Conditions in Theorem \ref{existence_gener_lin_thm}}
Using the identities \eqref{how_to_integrate} and \eqref{constant_C_X,c} as well as the boundedness of \eqref{boundedness_of_integral_over_eta} we obtain a bound
	\begin{align*}
	&\rho_F(\mu_t, \nu_t) \leq  \sup_{\|\psi\|_{BL} \leq 1} \int_S \psi(X(t,0,x)) e^{\int_0^t c(s,\, X(s,\,0,\,x))  \diff  s}  \diff\!\left[\mu_0 - \nu_0\right](x)  \\
	 & \phantom{ = }+\sup_{\|\psi\|_{BL} \leq 1} \int_0^t \int_S \int_S \psi(X(t,\tau,y)) e^{\int_{\tau}^t c(s,\, X(s,\,\tau,\,y)) \diff  s}  \diff[\eta(\tau,x)](y)  \diff\!\left[\mu_{\tau} - \nu_{\tau}\right](x)  \diff\tau  \\
	&  \leq \, C^{X,c}_{0,t} \, \rho_F(\mu_0, \nu_0) +\int_0^t C^{X,c}_{\tau,t} \, \| \eta(\tau,\cdot) \|_{BL(S;\mathcal{M}^+)} ~\rho_F(\mu_{\tau}, \nu_{\tau})  \diff\tau.
	\end{align*}
Note that this quantity is finite due Assumptions \ref{assumptions_general_model} and we can thus conclude the proof by applying Gr\"{o}nwall's inequality. 
\end{proofof}

\section{Proof of the nonlinear model}
\label{section:proof_nonlinear}

\begin{proofof}{Theorem \ref{uniq_theorem_nonlin_general}}
We aim at reducing the nonlinear model to a linear version where the measure argument is fixed. To this end, let $\mu_{\bullet}$ be in $C^0([0,T]; \mathcal{M}^+(S))$ and define the operator $T\mu_{\bullet} := \nu_{\bullet}$, where $\nu_{\bullet}$ is the unique solution of the linear equation
\begin{align*}
&\nu_t=X(t,0,\cdot,\mu_{\bullet})_{\#}\left(\mu_0(\cdot) e^{\int_0^t c(s,\,X(s,\,0,\cdot,\,\mu_{\bullet}), \mu_{s}) \diff  s}\right) \phantom{\int}  \\
& \, +\int_0^tX(t,\tau,\cdot, \mu_{\bullet})_{\#}\left(\int_S \left[\eta (\tau, y, \mu_{\tau}) (\cdot)\right]  \diff \nu_{\tau} (y) e^{\int_{\tau}^t c(s,\,X(s,\,\tau,\cdot,\, \mu_{\bullet}),\, \mu_{s}) \diff  s}\right) \diff\tau \\
& \, + \int_0^tX(t,\tau,\cdot, \mu_{\bullet})_{\#}\left(N(\tau, \mu_{\tau})(\cdot) e^{\int_{\tau}^t c(s,\,X(s,\,\tau,\cdot,\, \mu_{\bullet}), \,\mu_{s}) \diff  s}\right) \diff\tau. 
 \end{align*}
If the model functions $c(t,x,\mu_t)$, $\eta(t,x,\mu_t)(\cdot)$ and $N(t,\mu_t)(\cdot)$ are integrable with respect to time, then Theorem \ref{existence_gener_lin_thm} guarantees existence and uniqueness of $\nu_{\bullet}$. The continuity of the map $\mu_{\bullet}$ and the compactness of $[0,T]$ yield a bound $R>0$ such that $\sup_{t\in[0,T]}\|\mu_t\|_{BL^*}\leq R$ and consequently we can conclude the integrability condition from the uniform bounds \eqref{cond_suff_to_est_unif_bounds} in Assumption \ref{assumptions_general_model_non}.  As in the linear case one can check that $\nu_{\bullet}$ is in $C^0([0,T]; \mathcal{M}^+(S))$. Now define the set
    \begin{align*}
\mathcal{K}_{C} = \{\nu_{\bullet} \in C^0([0,T]; \mathcal{M}^+(S))\mid \nu_{0} = \mu_{0}, \sup_{t \in [0,T]} \| \nu_t \|_{BL^*} \leq C\}.
    \end{align*}
Choosing $C$ big enough, we can achieve $T: \mathcal{K}_{C} \to \mathcal{K}_{C}$ (see Lemma \ref{bound_on_TV_norm_gen_model} and the uniform bounds from condition \eqref{cond_suff_to_est_unif_bounds}).

Now, consider two solutions $\mu_{\bullet}^{(1)}, \mu_{\bullet}^{(2)} \in C^0([0,T]; \mathcal{M}^+(S))$ with identical initial condition $\mu_{0}$ and let $\nu_{\bullet}^{(1)} = T(\mu_{\bullet}^{(1)})_{\bullet}$ and $\nu_{\bullet}^{(2)} = T(\mu_{\bullet}^{(2)})_{\bullet}$. The continuity with respect to model functions (Theorem \ref{existence_gener_lin_thm}) and Assumption \ref{assumptions_general_model_non} with $R = C$ yields
	\begin{align*}
	&\rho_F\left(\nu^{(1)}_t, \nu^{(2)}_t\right) \leq  \,C_M \int_0^t \sup_{y \in S} \rho_F\left(\eta(\tau,y,\mu_{\tau}^{(1)}), \eta(\tau, y, \mu_{\tau}^{(2)})\right)  \diff\tau  \\ 
	&\phantom{ = }+ C_M \int_0^t \left[\rho_F\left(N(\tau, \mu_{\tau}^{(1)}), N (\tau, \mu_{\tau}^{(2)}) \right)  + \|c(\tau, \cdot, \mu_{\tau}^{(1)})-  c(\tau, \cdot, \mu_{\tau}^{(2)}) \|_{\infty}   \right] \diff\tau  \\
	&\phantom{ = }+C_M \sup_{0 \leq \tau_1 \leq \tau_2 \leq t} \|X(\tau_2, \tau_1, \cdot, \mu^{(1)}_{\bullet}) -  X(\tau_2, \tau_1, \cdot, \mu^{(2)}_{\bullet})  \|_{\infty}  \phantom{\int}\\ 
	&\leq C_M  \int_0^t \left(L_{R,\eta}(\tau) + L_{R,N}(\tau) + L_{R,c}(\tau)\right) \rho_F(\mu_{\tau}^{(1)}, \mu_{\tau}^{(2)})  \diff\tau \\
	&\phantom{ = } + C_M \sup_{0 \leq \tau_1 \leq \tau_2 \leq t} \int_{\tau_1}^{\tau_2} L_{R,X}(\tau)\, \rho_F(\mu_{\tau}^{(1)}, \mu_{\tau}^{(2)})  \diff\tau   \\ 
	 & \leq  \, C_M \sup_{0 \leq \tau_1 \leq \tau_2 \leq t} \int_{\tau_1}^{\tau_2} L_{R}(\tau)\, \rho_F(\mu_{\tau}^{(1)}, \mu_{\tau}^{(2)}) \diff\tau.
	\end{align*}
We apply the generalized Bielecki norm \eqref{L1Bielecki} once more with some $\lambda > 0$ and $f = L_R$ and see that
	\begin{align*}
	&\rho_{\lambda,L_R}(\nu^{(1)}_{\bullet}, \nu^{(2)}_{\bullet}) \leq \sup_{t \in [0,T]}\left[ e^{-\lambda \int_0^t L_R(u)  \diff u } \rho_F(\nu_t^{(1)}, \nu_t^{(2)}) \right]  \\
	&\leq C_M \sup_{t \in [0,T]} e^{-\lambda \int_0^t L_R(u)  \diff u } \sup_{0 \leq \tau_1 \leq \tau_2 \leq t} \int_{\tau_1}^{\tau_2} L_{R}(\tau) \rho_F(\mu_{\tau}^{(1)}, \mu_{\tau}
	 ^{(2)})  \diff\tau \\ 
	  &\leq C_M \, \rho_{\lambda,L_R}(\mu^{(1)}_{\bullet}, \mu^{(2)}_{\bullet}) \sup_{t \in [0,T]} e^{-\lambda \int_0^t \, L_R(u)  \diff u } \sup_{0 \leq \tau_1 \leq \tau_2 \leq t}\int_{\tau_1}
	  ^{\tau_2} L_{R}(\tau)\, e^{\lambda \int_0^{\tau} L_R(u)  \diff u }  \diff\tau.
	\end{align*}
For $0 \leq \tau_1 \leq \tau_2 \leq t$ we can estimate the last factor as
\begin{align*}
&\int_{\tau_1}^{\tau_2} L_{R}(\tau) e^{\lambda \int_0^{\tau} L_R(u)  \diff u }  \diff\tau = \frac{1}{\lambda} \left[e^{\lambda \int_0^{\tau_2} L_R(u)  \diff u } - e^{\lambda \int_0^{\tau_1} L_R(u)  \diff u }  \right] \\ &\leq 
\frac{1}{\lambda} \left[e^{\lambda \int_0^{t} L_R(u)  \diff u } - 1\right],
\end{align*}
so that we obtain the final estimate
$$
\rho_{\lambda,L_R}(\nu^{(1)}_{\bullet}, \nu^{(2)}_{\bullet}) \leq \frac{C_M}{\lambda} \rho_{\lambda,L_R}(\mu^{(1)}_{\bullet}, \mu^{(2)}_{\bullet}).
$$
In particular, for $\lambda$ large enough $T$ is contractive and thus we conclude existence and uniqueness by Banach Fixed Point Theorem. Finally, we note that the continuity estimates \eqref{cont_mf_gen_nonlinear_chap3} and \eqref{cont_ic_gen_nonlinear_chap3} follow directly from the corresponding results for the linear problem (see Theorem \ref{existence_gener_lin_thm}). 
\end{proofof}

\section{Proof of consistency with the PDE model}
\label{section:proof_of_consistency}
To establish the equivalence of the generalized model and the PDE model we prove that the representation formula for the generalized solution provides the unique measure solution to model \eqref{gen_model_Rd}.

 \begin{proofof}{Theorem \ref{reprimpliesmeasuresol}} As already suggested in the statement of the theorem we want to prove that the solution $\mu_{\bullet}$ of \eqref{NonlinearModelMetric}  also satisfies the weak formulation \eqref{Nonlin_weak_gen_ver_chap3_Rd}. Uniqueness then follows from Theorem \ref{thm:Rd_Main}.
 
We first note that by Remark \ref{rem:approximation_of_compact_support}, the test functions can actually assumed to be compactly supported in space. Additionally, by the density result (Theorem D.5 in Ref. \cite{our_book_ACPJ}, 
 it sufficient to consider test functions of the form
$f(t,x)=\varphi(t)\psi(x)$
with $\varphi \in C^1([0,T]) \cap W^{1,\infty}([0,T])$ and $\psi \in C^1_c(\mathbb{R}^d)$. We claim that the first part of the theorem will follow if we show that the map 
$$
[0,T] \ni t \mapsto \mathcal{F}(t) = \int_{\mathbb{R}^d} \psi(x)  \diff\mu_t(x)
$$
is differentiable a.e. in $t$  with derivative
	\begin{align}\label{formula_for_derivative_mang}
	&\mathcal{F}'(t) = \int_{\mathbb{R}^d} \nabla_x  \psi(x) \cdot b(t,x,\mu_t) + \psi(x) \, c(t,x,\mu_t)    \diff\mu_t(x) \nonumber  \\ 
	&\phantom{ = }+\int_{\mathbb{R}^d} \left(\int_{\mathbb{R}^d} \psi(y)  \diff\!\left[\eta(t,x,\mu_t) \right](y) \right)  \diff\mu_t(x) + 
	\int_{\mathbb{R}^d}  \psi(x)  \diff\!\left[N(t,\mu_t) \right](x). 
	\end{align}
Let us assume for the moment that \eqref{formula_for_derivative_mang} holds. In this case, Assumption \ref{assumptions_general_model_non} implies that $\mathcal{F}'$ is integrable and consequently $\mathcal{F}\in W^{1,1}([0,T])$, i.e. $\mathcal{F}$ is absolutely continuous. Now by the  Fundamental Theorem of Calculus and the product rule we see
\begin{align*}
&\int_{\mathbb{R}^d}f(T,x)\diff\mu_T(x)-\int_{\mathbb{R}^d}f(0,x)\diff\mu_0(x) =\,\varphi(T)\mathcal{F}(T)-\varphi(0)\mathcal{F}(0)\\
&=\int_0^T\partial_t[\varphi(t)\mathcal{F}(t)] \diff t =\int_0^T\partial_t\varphi(t)\mathcal{F}(t)+\varphi(t)\mathcal{F}'(t) \diff t,
 \end{align*}
which is exactly the right-hand side of  \eqref{Nonlin_weak_gen_ver_chap3_Rd} with test function $f(t,x)=\varphi(t)\psi(x)$.

We complete the proof by showing \eqref{formula_for_derivative_mang}. To this this end we introduce the following abbreviations
$$
\mathcal{X}_{t,\tau}(x) := X_b(t,\tau,x, \mu_{\bullet}), \qquad \qquad
\mathcal{E}_{t,\tau}(x):=e^{\int_{\tau}^t c(s, \,\mathcal{X}_{s,\tau}(x),\, \mu_{s})\diff s},
$$
which brings \eqref{how_to_integrate} in the form
	\begin{align}\label{eq:hti_withnewnotation}
	&\int_{\mathbb{R}^d} \psi(x)  \,\mathrm{d}\mu_t(x) = \int_{\mathbb{R}^d} \psi(\mathcal{X}_{t,0}(x)) \, \mathcal{E}_{t,0}(x)   \,\mathrm{d}\mu_0(x) \nonumber  \\
	&\phantom{ = }+\int_0^t \int_{\mathbb{R}^d} \int_{\mathbb{R}^d} \psi(\mathcal{X}_{t,\tau}(y)) \, \mathcal{E}_{t,\tau}(y)  \,\mathrm{d}[\eta(\tau,x,\mu_{\tau})](y)  \,\mathrm{d}\mu_{\tau}(x)  \,\mathrm{d}\tau \nonumber\\
	&\phantom{ = }+\int_0^t  \int_{\mathbb{R}^d} \psi(\mathcal{X}_{t,\tau}(x)) \, \mathcal{E}_{t,\tau}(x)  \,\mathrm{d}[N(\tau,\mu_{\tau})](x)   \,\mathrm{d}\tau.
	\end{align}
To prove formula \eqref{formula_for_derivative_mang} each term on the right- hand side of 	\eqref{eq:hti_withnewnotation}
has to be differentiated. We start with the first term and use \eqref{general_ODE_chap3} to compute
	\begin{align*}
& \frac{\mathrm{d}}{\mathrm{d}{t}} \left[\int_{\mathbb{R}^d} \psi(\mathcal{X}_{t,0}(x)) \, \mathcal{E}_{t,0}(x)   \,\mathrm{d}\mu_0(x) 
	\right] =\\
	&   = \int_{\mathbb{R}^d} \left[\nabla_{x}\psi(\mathcal{X}_{t,0}(x)) \cdot b(t,\mathcal{X}_{t,0}(x), \mu_t) \right] \mathcal{E}_{t,0}(x)   \,\mathrm{d}\mu_0(x)  \\ 
	&  \phantom{ = } +\int_{\mathbb{R}^d} \psi(\mathcal{X}_{t,0}(x)) \, \mathcal{E}_{t,0}(x)  \, c(t, \mathcal{X}_{t,0}(x), \mu_t)  \,\mathrm{d}\mu_0(x)  =: A + B.
	\end{align*}
The treatment of the second term of \eqref{eq:hti_withnewnotation} requires differentiation under the Bochner integral so that we have to apply the Dominated Convergence Theorem for Bochner integrals (see e.g. Theorem H.7 in Ref. \cite{our_book_ACPJ}) which yields 
\begin{align*}
	&\frac{\mathrm{d}}{\mathrm{d}{t}} \left[\int_0^t \int_{\mathbb{R}^d} \int_{\mathbb{R}^d} \psi(\mathcal{X}_{t,\tau}(y)) \, \mathcal{E}_{t,\tau}(y)  \,\mathrm{d}[\eta(\tau,x,\mu_{\tau})](y)  \,\mathrm{d}\mu_{\tau}(x)  \,\mathrm{d}\tau \right] \\ 
	&=  \int_{\mathbb{R}^d} \int_{\mathbb{R}^d} \psi(y)  \,\mathrm{d}[\eta(t,x,\mu_{t})](y)  \,\mathrm{d}\mu_{t}(x)\\ 
	&\phantom{ = }+ \int_0^t \int_{\mathbb{R}^d} \int_{\mathbb{R}^d} \left[\nabla_y \psi(\mathcal{X}_{t,\tau}(y)) \cdot b(t, \mathcal{X}_{t,\tau}(y),\mu_t) \right] \, \mathcal{E}_{t,\tau}(y)  \,\mathrm{d}[\eta(\tau,x,\mu_{\tau})](y)  \,\mathrm{d}\mu_{\tau}(x)  \,\mathrm{d}\tau  \\ 
	&\phantom{ = }+\int_0^t \int_{\mathbb{R}^d} \int_{\mathbb{R}^d} \psi(\mathcal{X}_{t,\tau}(y)) \, \mathcal{E}_{t,\tau}(y) \, c(t, \mathcal{X}_{t,\tau}(y),\mu_t)  \,\mathrm{d}[\eta(\tau,x,\mu_{\tau})](y)  \,\mathrm{d}\mu_{\tau}(x)  \,\mathrm{d}\tau   \\
	&=: C + D + E.\phantom{\int}
	\end{align*}
The third term in \eqref{eq:hti_withnewnotation} can be computed analogously to be 
	\begin{align*}
	&\frac{\mathrm{d}}{\mathrm{d}{t}} \left[\int_0^t \int_{\mathbb{R}^d} \psi(\mathcal{X}_{t,\tau}(x)) \, \mathcal{E}_{t,\tau}(x)   \,\mathrm{d}\!\left[N(\tau, \mu_{\tau}) \right](x)  \,\mathrm{d}\tau \right] =  \\ 
	  &=  \int_{\mathbb{R}^d} \psi(x)  \,\mathrm{d}[N(t,\mu_{t})](x) \\
	 &\phantom{ = }+ \int_0^t \int_{\mathbb{R}^d}  \left[\nabla_x \psi(\mathcal{X}_{t,\tau}(x)) \cdot b(t, \mathcal{X}_{t,\tau}(x),\mu_t) \right] \, \mathcal{E}_{t,\tau}(x)  \,\mathrm{d}[N(\tau,\mu_{\tau})](x)   \,\mathrm{d}\tau  \\ 
	&\phantom{ = }+\int_0^t \int_{\mathbb{R}^d} \psi(\mathcal{X}_{t,\tau}(x))\, \mathcal{E}_{t,\tau}(x) \, c(t, \mathcal{X}_{t,\tau}(x),\mu_t)  \,\mathrm{d}[N(\tau, \mu_{\tau})](x)  \,\mathrm{d}\tau   \\
	 &=: F + G + H.\phantom{\int}
	\end{align*}
Collecting all terms involving $\nabla_x \psi\cdot b$ and using formula \eqref{how_to_integrate} leads to the following simplification
	\begin{align}
	\label{eq:split_up_part_1}
	A + D + G = \int_{\mathbb{R}^d} \nabla_x \psi(x) \cdot b(t,x,\mu_t)  \diff\mu_t(x),
	\end{align}
and similarly for all terms with $\psi\, c$
	\begin{align}
	\label{eq:split_up_part_2}
	B + E + H = \int_{\mathbb{R}^d} \psi(x) \, c(t,x,\mu_t)  \diff\mu_t(x).
	\end{align}
Thus, formular \eqref{formula_for_derivative_mang} follows directly if we combine \eqref{eq:split_up_part_1} and \eqref{eq:split_up_part_2} with the remaining terms $C$ and $F$.
\end{proofof}

\begin{remark}
\label{rem:approximation_of_compact_support}
We shortly explain why it is sufficient to consider test functions which are compactly supported in space. Take a test function $\varphi\in C^1([0,T]\times \mathbb{R}^d)\cap W^{1,\infty}([0,T]\times\mathbb{R}^d)$ and let $\eta\in C_c^{\infty}(\mathbb{R}^d)$ be a smooth cut-off function with $\eta\equiv 1$ on $B_1(0)$, $0\leq \eta\leq 1$ and $\mathrm{supp}(\eta)\subseteq B_2(0)$. Then for $n\in \mathbb{N}$ we define sequence of functions $\varphi_n(t,x)=\varphi(t,x)\,\eta\left( \frac x n\right)$, where each component of $x$ is multiplied with the scalar $1/n$. By construction $\varphi_n\to \varphi$ for $n\to \infty$ pointwise and by Dominated Convergence also in $L^1$. Consequently, the only term in the weak formulation  \eqref{Nonlin_weak_gen_ver_chap3_Rd} which needs special attention is the term involving $\nabla_x \varphi$. Using Assumption \ref{ass_special_caseRd}, narrow continuity of $\mu_{\bullet}$ and dominated convergence yields
  \begin{align*}
 & \left\|\int_0^T \int_{\mathbb{R}^d} \left[\nabla_x  \varphi_n(t,x)-\nabla_x\varphi(t,x)\right] \cdot b(t,x,\mu_t)\diff\mu_t(x)  \diff t\right\|\to 0 \qquad (n\to \infty),
 \end{align*}
so that the approximation is indeed valid in the limit.
\end{remark}

\section{More on Measure Theory}
\label{appendix: measure theory}

Here we provide a comprehensive and complete summary of the measure theoretic results which form the theoretical background of this paper. Unless stated otherwise, the spaces $\mathcal{M}(S)$ and $\mathcal{M}^+(S)$ will always be equipped with the flat metric and all topological properties refer to the corresponding topology.

\begin{definition}
\label{defn:tightnesssequence}
 \begin{itemize}
     \item [(i)] For a measure $\mu\in \mathcal{M}(S)$ with corresponding Jordan decomposition $(\mu^+,\mu^-)$ (cf. Theorem 3.4 in Ref. \cite{Folland.1984}), the \textbf{variation} $|\mu|$ of $\mu$ is defined by
	\begin{align*}
		|\mu(A)|:=\mu^+(A)+\mu^-(A) \hspace{1cm} \forall A\in  \mathcal{B}(S).
	\end{align*}
    Note that the \textbf{total variation norm} is given by $\|\mu\|_{TV}=|\mu(S)|$.
     \item [(ii)] A set $\mathcal{N}\subseteq \mathcal{M}(S)$ is called \textbf{tight} if
for each $\varepsilon>0$ there exists a compact set $K\subseteq S$ with
	\begin{align*}
		|\mu|(S\setminus K)<\varepsilon \hspace{0.5cm} \forall \mu \in \mathcal{N},
	\end{align*}
	In particular, a sequence $(\mu_n)_{n\in \mathbb{N}} \subset \mathcal{M}(S)$ is \textbf{tight}, if $\mathcal{N}:=\{\mu_n\mid  n\in \mathbb{N}\}$ is tight and a single measure $\mu$ is \textbf{tight}, if $\{\mu\}$ is tight.
 \end{itemize}
\end{definition}

\begin{theorem}
\label{thm:M(S)_separablel}
Let $(S, d)$ be separable. Then the spaces $\mathcal{M}(S)$ and $\mathcal{M}^+(S)$ are separable with countable dense set given by the linear span $\mathrm{lin}\{\delta_s\}_{s\in S}$.
\end{theorem}

The main idea of the proof is to show that the linear span of Dirac measures concentrated in the dense subset of $S$ approximates all measures in $\mathcal{M}(S)$ in the flat metric sufficiently well. By restricting the weights to be rational, the approximation set becomes countable while still being dense. Further restricting the weights to be nonnegative yields that  $\mathcal{M}^+(S)$ is also separable. 

The next theorem characterizes the predual space of the bounded Lipschitz function.
\begin{theorem}
\label{thm:Predual_bounded_Lipschitz}
Let $(S,d)$ be separable and denote  $E := \overline{\mathcal{M}(S)}^{\|\cdot\|_{BL^*}}$. Then it holds that
	$E^*=BL(S)$.
\end{theorem}
 
In view of Banach Fixed Point Theorem the completeness of the measure space has been crucial for the definition or rather the domain of our operator, see
\eqref{eq:natural_metric_cmeas}. Unfortunately, Theorem 1.36 in Ref. \cite{our_book_ACPJ} states 
that $E=\mathcal{M}(S)$ if and only if $S$ is uniformly discrete, i.e. if $\inf_{\substack{x\neq y\\x,y\in S}}d(x,y)>0$.  As we don't want to restrict ourselves to uniformly discrete metric spaces, we will concentrate on the cone $\mathcal{M}^+(S)$ to achieve completeness. 

\begin{theorem}
\label{thm:M+_complete_wrt_flat}
Let $(S,d)$ be a Polish metric space. Then every Cauchy sequence $(\mu_n)_{n\in \mathbb{N}}\subset \mathcal{M}^+(S)$ converges in $\mathcal{M}^+(S)$ and in particular $\mathcal{M}^+(S)$ is complete with respect to the flat norm.
\end{theorem}

 The next theorem shows that the flat norm metrizes the narrow convergence on the cone. For the proof, we refer to Theorem 1.57 in Ref. \cite{our_book_ACPJ}.
\begin{theorem}
\label{thm: flat norm equivalent to narrow convergence}
Let $(S,d)$ be a Polish metric space. Then $\|\cdot\|_{BL^*}$ metricizes the narrow topology on $\mathcal{M}^+(S)$, i.e. we have 
	\begin{align*}
     \mu_n \to \mu \text{ narrowly }  
     \Longleftrightarrow 
		 \|\mu_n-\mu\|_{BL^*}\to 0.
	\end{align*}
\end{theorem}

The last result concerns the characterization of positive functionals on $BL(S)$. These functionals appear naturally in limiting procedures such as the construction of the Bochner integral on $E$ in Proposition \ref{proposition:rigorous definition of integrals}. As in general $\mathcal{M}(S)\subsetneq E$, not all elements in $E$ are measures which is potentially problematic in terms of interpretability of solutions. Fortunately, it turns out that positive functionals are actually nonnegative measures \cite{HilleWorm}.

\begin{definition}
\label{def:BL^*_+}
The space of \textbf{positive linear functionals} on $BL(S)$ is given by
	\begin{align*}
		BL(S)^*_+:=\{T\in BL(S)^*\mid T(\psi)\geq 0 \,\forall \psi\in BL(S),\psi\geq 0\}.
	\end{align*}
\end{definition}

\begin{theorem}\label{satz:respresentation of positive functionals}
Let $(S,d)$ be a Polish metric space. Then
	\begin{align*}
		E\cap BL(S)_+^*=\mathcal{M}^+(S).
	\end{align*}
\end{theorem}

\subsection{Proof of Theorem \ref{thm:Predual_bounded_Lipschitz}}

\begin{proof}
Consider a linear space $D \subseteq \mathcal{M}(S)$ spanned by finite linear combinations of Dirac deltas, i.e. 
	\begin{align*}
	D=\mathrm{lin}\{\delta_s\mid s\in S\}=\left\{\sum_{k=1}^n\alpha_k\delta_{s_k}\mid n\in\mathbb{N}, \alpha_k\in \mathbb{R}, s_k\in S\right\}
	\end{align*}
and let $D^{\ast}$ be its dual space, i.e. the space of linear functionals bounded with respect to the operator norm corresponding to $\|\cdot\|_{BL^*}$. As a first step we claim that we can identify $D^{\ast}$ with $BL(S)$. The conclusion will then follow by density of $D$ in $\mathcal{M}(S)$.

To prove the claim, for $f\in BL(S)$ define the linear and bounded functional $T_f\in D^*$  via
	\begin{align*}
	T_f:D\to \mathbb{R},\qquad T_f(\mu)=\int_Sf\diff\mu .
	\end{align*}
Now consider a functional $T \in D^{\ast}$. Due to linearity $T$  is uniquely defined by the values $T(\delta_s)$ for $s\in S$ and thus we identify $T$ with the function
	\begin{align*}
	f:S\to \mathbb{R},\qquad s\mapsto T(\delta_s).
	\end{align*}
We claim that $f\in BL(S)$. Assume by way of contradiction that $f$ was not bounded so that there exists a sequence $(s_n)_{n\in\mathbb{N}}$ with $\underset{n\to \infty}{\lim} f(s_n)= \infty$. But then
	\begin{align*}
		\lim_{n \to \infty}T(\delta_{s_n})=\lim_{n\to \infty} f(s_n)=\infty,
	\end{align*}
which contradicts the boundedness of the functional. On the other hand if $f$ was not Lipschitz continuous we can assume the existence of sequences $(s_n)_{n\in \mathbb{N}}, (t_n)_{n\in \mathbb{N}}$ such that
	\begin{align*}
		\lim_{n \to \infty} \frac {f(s_n)-f(t_n)}{d(s_n,t_n)}=\infty.
	\end{align*}
Define the sequence of measures $(\mu_n)_{n\in \mathbb{N}}$ with $\mu_n=\frac{\delta_{s_n}-\delta_{t_n}}{d(s_n,t_n)}$ and note that 
	\begin{align*}
		T(\mu_n)=\frac{T(\delta_{s_n})- T(\delta_{t_n})}{d(s_n,t_n)}= \frac {f(s_n)-f(t_n)}{d(s_n,t_n)}\to \infty.
	\end{align*}
However, this contradicts the boundedness of $T$ as for all $n\in \mathbb{N}$ 
	\begin{align*}
		\|\mu_n\|_{BL^*}=\sup_{\|\psi\|_{BL}\leq 1}\frac{\psi(s_n)- \psi(t_n)}{d(s_n,t_n)} \leq 1.
	\end{align*}
So we can indeed identify $D^{\ast}$ with $BL(S)$. By Theorem \ref{thm:M(S)_separablel}, $D$ is dense in $E$ and we can thus extend any  $T\in D^*$ uniquely to a functional in $\tilde T\in E^*$ (independent of the approximating sequence) completing the proof.
\end{proof}

\subsection{Proof of Theorem \ref{thm:M+_complete_wrt_flat}}

\begin{theorem}
\label{satz:M+boundedsubsequence} Let $(S,d)$ be a Polish metric space  and let $(\mu_n)_{n\in\mathbb{N}}$ be a tight sequence in $\mathcal{M}^+(S)$ with $\underset{n\in\mathbb{N}}{\sup}\,\mu_n(S)<\infty$. Then $(\mu_n)_{n\in \mathbb{N}}$ has a converging subsequence in $\mu\in \mathcal{M}^+(S)$.
\end{theorem}

\begin{proof}
Applying the Theorem of Prokhorov (see Theorem 2.3 in Ref. \cite{Da_Prato_2014}) yields a narrowly converging subsequence with limit in $\mathcal{M}^+(S)$. By Theorem  \ref{thm: flat norm equivalent to narrow convergence} narrow convergence and convergence with respect to the flat metric are equivalent.
\end{proof}

\begin{remark}
\label{rem:proper}
If $S$ is compact, then Theorem \ref{satz:M+boundedsubsequence} implies that every bounded and
closed subset of $\mathcal{M}^+(S)$ is compact. This property is called \textbf{proper}.
\end{remark}

On arbitrary Borel subsets of $S$ it is in general not possible to conclude a triangular-type inequality for the flat metric. The following lemma emphasizes this while simultaneously providing a pseudo triangle inequality.
\begin{lemma}
\label{lem:set_inequality_flat}
Let $\delta \in (0,1]$, $T\in \mathcal{B}(S)$. Then for $\mu,\nu \in \mathcal{M}^+(S)$
	\begin{align*}
		&\mu(T)\leq\nu(U_{\delta}(T))+\frac 1 {\delta}\, \|\mu-\nu\|_{BL^*},\\
		&\mu(S\setminus U_{\delta}(T))\leq\nu(S\setminus T)+\frac 1 {\delta}\, \|\mu-\nu\|_{BL^*}.
	\end{align*}
Here $U_{\delta}(T):=\{x\in S\mid d(x,T)< \delta\}$ is the $\delta$-neighborhood of $T$.
\end{lemma}

\begin{proof}
According to Lemma 2.1 in Ref. \cite{GMT_Measures_under_Flat_Norm}, there is a Lipschitz continuous function $\eta:S\to [0,1]$ with $|\eta|_{\Lip}= 1/ {\delta}$ and such that $\eta \equiv 1$ on $T$ and $\eta \equiv 0$ on $S\setminus U_{\delta}(T)$. In particular, $\|\eta\|_{BL}\leq 1/\delta$. Thus,
	\begin{align*}
		\mu(T)
  \leq \int_S  \eta \diff\mu
  \leq \frac 1 {\delta}\ \|\mu-\nu\|_{BL^*}+\int_S \! \eta  \diff\nu
		\leq\frac 1 {\delta}\ \|\mu-\nu\|_{BL^*}+\nu(U_{\delta}(T)).
	\end{align*}
The other equation is proven similarly by choosing a function $\tilde \eta:S\to [0,1]$ such that $\tilde \eta \equiv 1$ on $S\setminus U_{\delta}(T)$ and $\tilde \eta\equiv 0$ on $T$.
\end{proof}

\begin{proposition}\label{prop:flat17feb4.27}
Let $(S,d)$ be a Polish metric space. If $\mathcal{N}\subseteq \mathcal{M}^+(S)$ is totally bounded with respect to $\|\cdot\|_{BL^*}$, then $\mathcal{N}$ is tight.
\end{proposition}

\begin{proof}
We follow the reasoning from Ref. \cite{GMT_Measures_under_Flat_Norm}. Let $\mathcal{N}\subseteq \mathcal{M}^+(S)$ be a totally bounded set of measures. The total boundedness implies that we can discretize $\mathcal{N}$, i.e. for $\varepsilon \in (0,1)$, there exists a finite subset $\tilde{\mathcal{N}} \subseteq \mathcal{N}$ such that 
	\begin{align}
		\label{eq:flat17feb4.27}
		\mathcal{N} \subseteq \bigcup_{\nu \in \tilde{\mathcal{N}}} U_{{\varepsilon^2}/ 4}(\nu)=:U_{{\varepsilon^2}/ 4}
		(\tilde{\mathcal{N}}).
	\end{align}
As $S$ is separable and complete, $\{\mu\}$ is tight for any finite measure $\mu$, cf. Remark 13.27 in Ref. \cite{Klenke.2020}. It follows that the finite set $\tilde{\mathcal{N}}$ is tight. Hence, there exists a compact set $K\subseteq S$  such that $\nu(S\setminus K)< {\varepsilon}/2$ for all $\nu \in \tilde{\mathcal{N}}$.

In a next step, we discretize $K$. The compactness yields that there exists a finite set of points $F=F_{\varepsilon}\subset K$ such that $K\subseteq U_{ {\varepsilon}/ 2} (F)$. Note that still
    \begin{align*}
    \nu(S \setminus U_{ {\varepsilon}/ 2}(F))\leq \nu(S\setminus K)<{\varepsilon}/2 \qquad\forall\nu\in \tilde{\mathcal{N}}.
    \end{align*}
Now let $\mu \in \mathcal{N}$ be arbitrary with corresponding measure $\nu \in \tilde{\mathcal{N}}$ such that $\|\mu-\nu\|_{BL^*}\leq  {\varepsilon^2}/ 4$. Furthermore, observe that $U_{{\varepsilon}/2}(U_{ {\varepsilon}/ 2}(F))\subseteq U_{\varepsilon}(F)$ by triangle inequality. Hence, by Lemma \ref{lem:set_inequality_flat} 
	\begin{align*}
		\mu(S\setminus U_{\varepsilon}(F))\leq\mu(S\setminus U_{{\varepsilon}/2}(U_{ {\varepsilon}/ 2}(F)))\leq \nu(S\setminus
		 U_{{\varepsilon}/2}(F))+ \frac 2{\varepsilon}\, \|\mu-\nu\|_{BL^*}<\varepsilon.
	\end{align*}
 In particular, for each $\varepsilon \in (0,1)$ we constructed a finite set $F_{\varepsilon}\subseteq S$ such that $\mu(S\setminus U_{\varepsilon}(F_{\varepsilon}))<\varepsilon$ for all $\mu \in \mathcal{N}$.
 
We conclude the proof by using the family of sets $\{F_{\varepsilon}\}_{\varepsilon<1}$ to construct a suitable compact set to prove tightness of $\mathcal{N}$. Let $\delta \in (0,1]$ be arbitrary and for $n \in \mathbb{N}$ set $\varepsilon_n=\delta 2^{-n}$. With the considerations above we obtain for each $n\in \mathbb{N}$ a finite set $F_n:=F_{\varepsilon_n}$ such that
	\begin{align*}
		\mu(S\setminus U_{\varepsilon_n}(F_n))<\varepsilon_n \hspace{0.5cm}\forall\mu \in \mathcal{N}.
	\end{align*}
The infinite intersection of closed sets $K_{\delta}:={\bigcap}_{n\in \mathbb{N}}\overline{U_{\varepsilon_n}(F_n)}$ is closed and totally bounded by construction, i.e. compact. Furthermore, we have for all $\mu \in \mathcal{N}$
	\begin{align*}
		\mu(S\setminus K_{\delta})
		\leq\mu\left(\bigcup_{n\in \mathbb{N}}S		\setminus {U}_{\varepsilon_n}(F_n)\right)\leq
		\sum_{n=1}^{\infty}\mu(S\setminus {U}_{\varepsilon_n}(F_n))<\sum_{n=1}^{\infty}\varepsilon_n=\delta \sum_{n=1}^{\infty} 2^{-n}	
		=\delta.
	\end{align*}
In particular, $\mathcal{N}$ is tight.
\end{proof}

\begin{proofof}{Theorem \ref{thm:M+_complete_wrt_flat}}
We start with the observation that the set $\mathcal{N}:=\{\mu_n\mid n\in \mathbb{N}\}$ is  totally bounded since $(\mu_n)_{n\in\mathbb{N}}$ is Cauchy by assumption, so that we can conclude the tightness of $\mathcal{N}$ from Proposition \ref{prop:flat17feb4.27}. As the real valued sequence $(\mu_n(S))_{n\in\mathbb{N}}$ is bounded, we apply Theorem \ref{satz:M+boundedsubsequence} to extract a converging subsequence. However, then the whole sequence has to converge as it is Cauchy.
\end{proofof}

\subsection{Proof of Theorem \ref{satz:respresentation of positive functionals}}

\begin{proof}
We follow the reasoning from Ref. \cite{HilleWorm}. First, we make a preliminary observation that the positivity and the linearity imply that any $T\in BL_+^*(S)$ is monotone, i.e. $T(\psi)\geq T(\varphi)$ for any $\psi,\varphi\in BL(S)$ with $\psi\geq \varphi$.

Concerning the proof, we immediately conclude that $\mathcal{M}^+(S)\subseteq E\cap BL(S)_+^*$. By way of contradiction  assume that there is $T \in E\cap BL(S)_+^*$ such that $T\notin\mathcal{M}^+(S)$. Let $\textbf{1}\in BL(S)$ be the constant function with value $1$. We claim that if an operator $T$ vanishes for $\textbf{1}$, i.e. $T(\textbf{1})=0$, then it vanishes already for for all $\psi\in BL(S)$. Indeed, as $T$ is linear and monotone, we compute
 \begin{align*}
		&|T(\psi)|=|T(\psi^+)-T(\psi^-)|\leq |T(\psi^+)|+|T(\psi^-)|=T(\psi^+)+T(\psi^-)\\
		&=T(\psi^++\psi^-)=T(|\psi|)\leq T(\|\psi\|_{\infty}\cdot \textbf{1})=\|\psi\|_{\infty}T(\textbf{1})=0,
	\end{align*} 
where $\psi^{\pm}=\max \{\pm \psi,0\}$. Hence, $T=0\in \mathcal{M}^+(S)$, which is a contradiction.

So for the rest of the proof we can assume $T(\textbf{1})>0$. According to Theorem \ref{thm:M+_complete_wrt_flat}, $\overline{\mathcal{M}^+(S)}^{\|\cdot\|_{BL^*}}=\mathcal{M}^+(S)$, in particular $\mathcal{M}^+(S)$ is a closed subset of $E$. Thus, the Hyperplane Separation Theorem implies that the closed and convex set $\mathcal{M}^+(S)$ can be strictly separated from $\{T\}\nsubseteq \mathcal{M}^+(S)$ in $E$, i.e. there is $f\in E^*=BL(S)$ (see Theorem \ref{thm:Predual_bounded_Lipschitz}) and $\alpha\in \mathbb{R}$ such that
	\begin{align}
		\label{eq:separatinghyperplane}
		\langle \mu,f\rangle\leq\alpha\, \quad \forall \mu\in\mathcal{M}^+(S) \hspace{0,5cm} \text {and} \hspace{0,5cm} \langle T, f\rangle= 
		T(f)>\alpha.
	\end{align}
As $T(\textbf{1})\delta_x\in\mathcal{M}^+(S)$ for all $x\in S$, we get
	\begin{align*}
		\alpha\geq\langle T(\textbf{1})\delta_x,f\rangle=T(\textbf{1})f(x) \text{ for all } x\in S 
	\end{align*}
and thus, $f\leq\frac{\alpha}{T(\textbf{1})}$. Using the positivity of $T$
	\begin{align*}
		T(f)\leq T\left(\frac{\alpha\textbf{1}}{T(\textbf{1})}\right)=\alpha
	\end{align*}
yields a contradiction to (\ref{eq:separatinghyperplane}) and hence it holds $E\cap BL(S)_+^*=\mathcal{M}^+(S)$.
\end{proof}

\bibliographystyle{abbrv}
\bibliography{sample}


\end{document}